\documentclass[12pt]{amsart}

\usepackage{amsmath}
\usepackage{amsfonts}
\usepackage{amssymb}
\usepackage{graphicx}
\usepackage{hyperref}
\usepackage{mathrsfs}
\usepackage{pb-diagram}
\usepackage{epstopdf}
\usepackage{amsmath,amsfonts,amsthm,enumerate,amscd,latexsym,curves,multibox}
\usepackage{bbm}
\usepackage[mathscr]{eucal}
\usepackage{epsfig,epsf,xypic,epic}

\newtheorem{theorem}{Theorem}
\newtheorem{lemma}[theorem]{Lemma}
\newtheorem{corollary}[theorem]{Corollary}
\newtheorem{prop}[theorem]{Proposition}
\newtheorem{defn}[theorem]{Definition}

\newtheorem{remark}[theorem]{Remark}
\newtheorem{conjecture}[theorem]{Conjecture}
\newtheorem{assumption}[theorem]{Assumption}

\numberwithin{equation}{section}

\newcommand{\integer}{\mathbb{Z}}
\newcommand{\rat}{\mathbb{Q}}
\newcommand{\real}{\mathbb{R}}
\newcommand{\N}{\mathbb{N}}
\newcommand{\Z}{\mathbb{Z}}
\newcommand{\R}{\mathbb{R}}
\newcommand{\C}{\mathbb{C}}
\newcommand{\cpx}{\mathbb{C}}

\newcommand{\bb}{\boldsymbol{b}}
\newcommand{\bv}{\boldsymbol{v}}
\newcommand{\bx}{\boldsymbol{x}}
\newcommand{\be}{\boldsymbol{e}}

\newcommand{\consti}{\mathbf{i}\,}
\newcommand{\conste}{\mathbf{e}}

\newcommand{\WT}[1]{\widetilde{#1}}
\newcommand{\proj}{\mathbb{P}}

\newcommand{\CM}{\mathcal{M}}
\newcommand{\CE}{\mathcal{E}}
\newcommand{\CL}{\mathcal{L}}
\newcommand{\CP}{\mathcal{P}}
\newcommand{\bD}{\mathcal{D}}
\newcommand{\BA}{\mathbb{A}}

\newcommand{\Dom}{\mathrm{Dom}}

\begin{document}

\title[Lagrangian Floer superpotentials and crepant resolutions]{Lagrangian Floer superpotentials and crepant resolutions for toric orbifolds}

\author[Chan]{Kwokwai Chan}
\address{Department of Mathematics\\ The Chinese University of Hong Kong\\ Shatin\\ Hong Kong}
\email{kwchan@math.cuhk.edu.hk}

\author[Cho]{Cheol-Hyun Cho}
\address{Department of Mathematical Sciences, Research institute of Mathematics\\ Seoul National University\\ San 56-1, Shinrimdong\\ Gwanakgu \\Seoul 47907\\ Korea}
\email{chocheol@snu.ac.kr}

\author[Lau]{Siu-Cheong Lau}
\address{Department of Mathematics\\ Harvard University\\ One Oxford Street\\ Cambridge \\ MA 02138\\ USA}
\email{s.lau@math.harvard.edu}

\author[Tseng]{Hsian-Hua Tseng}
\address{Department of Mathematics\\ Ohio State University\\ 100 Math Tower, 231 West 18th Ave. \\ Columbus \\ OH 43210\\ USA}
\email{hhtseng@math.ohio-state.edu}

\date{\today}

\begin{abstract}
We investigate the relationship between the Lagrangian Floer superpotentials for a toric orbifold and its toric crepant resolutions. More specifically, we study an open string version of the crepant resolution conjecture (CRC) which states that the Lagrangian Floer superpotential of a Gorenstein toric orbifold $\mathcal{X}$ and that of its toric crepant resolution $Y$ coincide after analytic continuation of quantum parameters and a change of variables. Relating this conjecture with the closed CRC, we find that the change of variable formula which appears in closed CRC can be explained by relations between open (orbifold) Gromov-Witten invariants. We also discover a geometric explanation (in terms of virtual counting of stable orbi-discs) for the specialization of quantum parameters to roots of unity which appears in Y. Ruan's original CRC \cite{Ruan06}. We prove the open CRC for the weighted projective spaces $\mathcal{X}=\proj(1,\ldots,1,n)$ using an equality between open and closed orbifold Gromov-Witten invariants. Along the way, we also prove an open mirror theorem for these toric orbifolds.
\end{abstract}

\maketitle

\tableofcontents

\section{Introduction}

The {\em crepant resolution conjecture} (abbreviated as CRC) \cite{Ruan06, Bryan-Graber09, CIT09, Coates-Ruan} has attracted a lot of attention in the last ten years, and much evidence has been found, especially in toric cases \cite{Perroni07, BGP08, BMP09, CIT09, Coates09, CIJ}. This conjecture arises from string theory: if $\mathcal{X}$ is a Gorenstein orbifold and $Y$ is a crepant resolution, then $\mathcal{X}$ and $Y$ correspond to two large radius limit points (or cusps) in the so-called {\em stringy K\"ahler moduli space} $\CM_A$ which parametrizes a family of topological string theories (A-model) whose chiral rings should be given by the small quantum (orbifold) cohomology ring of the corresponding target space near each cusp. Hence it is natural to expect that the quantum cohomology rings $QH_\mathrm{orb}^*(\mathcal{X})$ and $QH^*(Y)$ are closely related.

Ruan \cite{Ruan06} wrote down the first precise statement which asserts that $QH^*(Y)$ is isomorphic to $QH_\mathrm{orb}^*(\mathcal{X})$ after analytic continuation of the quantum parameters of $Y$ and specializing the exceptional ones to roots of unity. Later, Bryan and Graber \cite{Bryan-Graber09} proposed a significant strengthening of this, asserting that, if $\mathcal{X}$ satisfies a {\em Hard Lefschetz condition}, then even the {\em big} quantum cohomology rings are isomorphic after analytic continuation and specialization of quantum parameters. At around the same time, Coates, Iritani and Tseng \cite{CIT09} (see also Coates-Ruan \cite{Coates-Ruan}) presented a rather different and more general formulation of the conjecture using Givental's Lagrangian cones and symplectic formalism \cite{Coates-Givental07,Givental04}. Their conjecture is expected to hold even without the Hard Lefschetz assumption on $\mathcal{X}$. See Subsection \ref{closedCRC_vs_openCRC} and Conjecture \ref{closed_CRC} below for more details.

In this paper, we study how the Lagrangian Floer superpotential of a Gorenstein toric orbifold and that of its toric crepant resolution are related under analytic continuation of quantum parameters. We propose an open string version of the CRC in the toric case. A compact toric manifold $Y$ has a Landau-Ginzburg (LG) mirror \cite{hori00}, which can be constructed using Lagrangian Floer theory (due to Fukaya, Oh, Ohta and Ono \cite{FOOO1}). More precisely, the {\em Lagrangian Floer superpotential} $W^{LF}_Y$, which is part of the data in the LG mirror of $Y$, is defined by virtual counting of stable holomorphic discs in $Y$ bounded by Lagrangian torus fibers of the moment map. In general, the coefficients of $W^{LF}_Y$, which are generating functions of genus zero open Gromov-Witten (GW) invariants, are only formal power series with values in the Novikov ring $\Lambda_0:=\{\sum_{k=1}^\infty C_k T^{\lambda_k} \mid C_k\in\rat, \lambda_k\in\R_{\geq0}, \lim_{k\to\infty}\lambda_k=\infty \}$, where $T$ is a formal parameter. In case these formal power series are convergent, this produces a family of holomorphic functions on the algebraic torus $(\C^*)^n$ ($n=\mathrm{dim}(Y)$) over a neighborhood $U_Y$ of the cusp in $\CM_A$ corresponding to $Y$.

Recently, the second author and Poddar \cite{CP} developed Lagrangian Floer theory for moment map fibers in compact toric orbifolds. They classified all holomorphic orbifold discs in a compact toric orbifold $\mathcal{X}$ bounded by these Lagrangian tori and defined open orbifold GW invariants by virtual counting of stable holomorphic orbi-discs. In particular, they defined a Lagrangian Floer superpotential $W$ using the counting of smooth holomorphic discs, and also a bulk deformed superpotential $W^{\frak b}$. The latter is defined by the virtual counting of stable orbifold discs where the bulk deformation $\frak b$ (i.e. insertion at interior orbifold marked points) is given by fundamental classes of twisted sectors.

In this paper, we define the Lagrangian Floer superpotential $W^{LF}_\mathcal{X}$ of $\mathcal{X}$, which is different from $W$ or $W^{\frak b}$, as a formal power series whose coefficients are generating functions of certain open orbifold GW invariants. Assuming convergence, this gives a family of holomorphic functions on $(\C^*)^n$ over a neighborhood $U_\mathcal{X}$ of the cusp in $\CM_A$ corresponding to $\mathcal{X}$.

We can now state our open CRC (same as Conjecture \ref{open_CRC}):
\begin{conjecture}[Open Crepant Resolution Conjecture]
Let $X$ be a toric variety with at worse Gorenstein quotient singularities. Let $\mathcal{X}$ be the canonical toric orbifold with $X$ as its coarse moduli space (see \cite[Section 7]{BCS}). And let $Y$ be a toric crepant resolution of $X$.  The flat coordinates on the K\"ahler moduli of $\mathcal{X}$ and $Y$ are denoted as $q$ and $Q$ respectively.  Let $l$ be the dimension of the K\"ahler moduli of $\mathcal{X}$ (which is equal to that of $Y$).

The Lagrangian Floer superpotential $W^{LF}_\mathcal{X}(q)$ of $\mathcal{X}$ is a Laurent series over the Novikov ring in $q$.  Similarly the Lagrangian Floer superpotential $W^{LF}_Y(Q)$ of $Y$ is a Laurent series over the Novikov ring in $Q$.  Then there exists
\begin{enumerate}
\item $\epsilon>0$;
\item a coordinate change $Q(q)$, which is a holomorphic map $(\Delta(\epsilon)-\real_{\leq 0})^{l} \to (\cpx^\times)^l$, and $\Delta(\epsilon)$ is an open disc of radius $\epsilon$ in the complex plane;
\item a choice of analytic continuation of coefficients of the Laurent series $W^{LF}_Y(Q)$ to the target of the holomorphic map $Q(q)$,
\end{enumerate}
such that $W^{LF}_Y(Q(q))$ defines a holomorphic family of Laurent series over a small neighborhood of $q = 0$, and
\begin{equation*}
W^{LF}_\mathcal{X}(q)=W^{LF}_Y(Q(q)).
\end{equation*}
\end{conjecture}

Indeed the above conjecture is part of the global picture given by the stringy K\"ahler moduli which is not mathematically defined yet.  The stronger conjectural global statement (for toric varieties) may be formulated as follows: There exists
\begin{enumerate}
\item a manifold $\CM_A$, the so-called stringy K\"ahler moduli;
\item a holomorphic family of Laurent series $W^{LF}$ over $\CM_A$;
\item a coordinate patch $(U_\mathcal{X},q)$ of $\CM_A$ such that $q^*(W^{LF})$ equals to the Lagrangian Floer superpotential of $\mathcal{X}$;
\item a coordinate patch $(U_Y,Q)$ of $\CM_A$ such that $Q^*(W^{LF})$ equals to the Lagrangian Floer superpotential of $Y$.
\end{enumerate}

Since we do not have a global construction of the stringy K\"ahler moduli space $\CM_A$ and also the chiral rings over points far away from the cusps, analytic continuation is required in all the crepant resolution conjectures. In practice, in order to prove the open or closed CRC, one first constructs the $B$-model moduli space $\CM_B$ (in toric cases, this is simply given by the toric orbifold associated to the secondary fan of the crepant resolution $Y$). Mirror symmetry will identify the neighborhoods $U_\mathcal{X}$ and $U_Y$ with neighborhood of certain cusps in $\CM_B$. Since the $B$-model moduli space is global, one can then perform analytic continuation over $\CM_B$, and (by applying mirror symmetry again) obtain the change of variables.

A remarkable feature of our open CRC is that it predicts equalities between generating functions of open GW invariants for $\mathcal{X}$ and $Y$ after analytic continuation and a change of variable. See the equalities \eqref{openmatch1}, \eqref{openmatch2} and the discussion after Conjecture \ref{open_CRC} at the end of Subsection \ref{sec:formulation_openCRC}.

Our open CRC also sheds new light on the study of the closed CRC. First of all, one may infer from our open CRC that the change of variable formula needed in the closed CRC actually originates from the geometric data of open GW invariants of an orbifold and its crepant resolution (by the equalities \eqref{openmatch1}, \eqref{openmatch2}). Furthermore, we discover a geometric explanation for the specialization of quantum parameters to roots of unity which appeared in Ruan's conjecture. Namely, we show that the specialization corresponds precisely to the vanishing of coefficients of $W^{LF}_Y$ which count stable holomorphic discs meeting the exceptional divisors in $Y$. See Theorem \ref{thm_specialization} for the precise statement and Subsection \ref{sec:specialization} for more details.

Indeed, we expect that the open and closed crepant resolution conjectures are closely related to each other since the Jacaobian ring of the Lagrangian Floer superpotential should be isomorphic to the small quantum cohomology ring. For toric manifolds, this was proved by Fukaya, Oh, Ohta and Ono in their recent work \cite{FOOO10b}:\footnote{In fact, they proved a much stronger result: the big quantum cohomology ring of $Y$ is isomorphic to the Jacobian ring of the {\em bulk-deformed} Lagrangian Floer superpotential.}
$$QH^*(Y)\cong Jac(W^{LF}_Y).$$
We plan to investigate the analogous story on the orbifold side in a subsequent work. What we expect to be true is the following:
\begin{conjecture} \label{conj:iso}
There is an isomorphism
$$QH_\mathrm{orb}^*(\mathcal{X})\cong Jac(W^{LF}_\mathcal{X}).$$
\end{conjecture}
Combining these two isomorphisms with the open CRC, we conclude that
$$QH^*(Y)\cong QH_\mathrm{orb}^*(\mathcal{X}),$$
via analytic continuation in quantum parameters and a change of variables. If we specialize the exceptional parameters to suitable values (not necessarily roots of unity), this will imply the ``quantum corrected" version of Ruan's conjecture as formulated by Coates and Ruan \cite{Coates-Ruan}. See Subsection \ref{sec:specialization} below for more details.

By the recent result \cite[Theorem 1.16]{G-W} of Gonzalez and Woodward, the quantum cohomology ring of $\mathcal{X}$ is isomorphic to the (appropriately defined) Jacobian ring of the potential function $W^{HV}$ defined in Definition \ref{W^HV} below. Therefore Conjecture \ref{conj:iso} should follow from an open mirror theorem (see Conjecture \ref{open_mirror_thm} below), which compares $W^{HV}$ and $W^{LF}$. Alternatively, we expect that Conjecture \ref{conj:iso} can be proven by following the strategy of \cite{FOOO10b}.

In this paper, we will prove the open CRC for the weighted projective spaces $\mathcal{X}=\proj(1,\ldots,1,n)$:
\begin{theorem}[=Theorem \ref{open_CRC_wPn}]\label{open_CRC_wPn_intro}
For the weighted projective space $\mathcal{X}=\proj(1,\ldots,1,n)$ whose crepant resolution is given by $Y=\proj(K_{\proj^{n-1}}\oplus\mathcal{O}_{\proj^{n-1}})$, the open CRC is true.
\end{theorem}
We prove this by first establishing a formula relating open and closed orbifold GW invariants for Gorenstein toric Fano orbifolds (Theorem \ref{THM_open_closed}); this generalizes the formula in \cite{Chan10,LLW10} to the orbifold setting. Then, we use the toric orbifold mirror theorem (for closed theories) recently proved by Coates, Corti, Iritani and Tseng \cite{CCIT13} to deduce an open toric mirror theorem for $\proj(1,\ldots,1,n)$ (Theorem \ref{open_mirror_thm_wPn}) and at the same time establish the convergence of the Lagrangian Floer superpotential $W^{LF}_\mathcal{X}$. We expect that this open toric mirror theorem (Conjecture \ref{open_mirror_thm}), which is an orbifold version of the one formulated in Chan-Lau-Leung-Tseng \cite{CLLT11}, is in general true for any compact toric K\"ahler orbifold (see Subsection \ref{sec:open_mirror_thm}).

Now the open CRC follows from this open mirror theorem and analytic continuation of the mirror maps for $\mathcal{X}$ and $Y$ (the convergence of the Lagrangian Floer superpotential $W^{LF}_Y$ is already proved in \cite{CLLT12}). We remark that the analytic continuation process was also done in the construction of the symplectic transformation $\mathbb{U}$ which appeared in Coates-Iritani-Tseng's formulation of the closed CRC \cite{CIT09}. We will discuss how the open toric mirror theorem is related to the open CRC in general (see Subsection \ref{closedCRC_vs_openCRC}).

Our strategy for proving Theorem \ref{open_CRC_wPn_intro} above is expected to work more generally in all semi-Fano cases. More precisely, we consider a compact simplicial toric variety $X$ which is semi-Fano in the sense of Definition \ref{defn:semiFano}. In this case  the canonical toric orbifold $\mathcal{X}$ is also semi-Fano. If $Y$ is a toric crepant resolution of $X$, then $Y$ is also semi-Fano. The strategy may be summarized in the following diagram:
\begin{center}
\setlength{\unitlength}{.5cm}
\begin{picture}(10,8) (1,-1)
\thicklines
\put(0,0){$W_\mathcal{X}^{LF}(q)$}
\put(10,0){$W_\mathcal{X}^{HV}(y)$}
\put(0,5){$W_Y^{LF}(Q)$}
\put(10,5){$W_Y^{HV}(U)$}
\put(3.3,.3){\line(1,0){6.5}}
\put(3.3,5.3){\line(1,0){6.5}}
\put(1,1.3){\line(0,1){3}}
\put(11,1.3){\line(0,1){3}}
\put(5,6.4){\scriptsize open }
\put(3.8,5.7){\scriptsize mirror theorem }
\put(4,1.4){\scriptsize orbifold open }
\put(3.8,0.7){\scriptsize mirror theorem }
\put(-4.0,3.1){\scriptsize open crepant}
\put(-5.1,2.4){\scriptsize resolution conjecture }
\put(12.6,3.1){\scriptsize  }
\put(11.5,2.4){\scriptsize   }
\end{picture}
\end{center}
On the right hand side we have the {\em Hori-Vafa superpotentials} $W_Y^{HV}$ and $W_\mathcal{X}^{HV}$ which are combinatorial in nature, see Definitions \ref{W^HV} and \ref{W^{HV}_Y}. On the top part of the diagram, the open mirror theorem for compact semi-Fano toric manifolds (Theorem \ref{open_mirror_thm_Y}\footnote{Theorem \ref{open_mirror_thm_Y} was first proposed and proved under a convergence assumption in \cite{CLLT11}, and was later proved {\em unconditionally} by an entirely different and much more geometric method in \cite{CLLT12}.}), relates $W^{LF}_Y$ and $W^{HV}_Y$:
$$W^{LF}_Y(Q)=W^{HV}_Y(U(Q)),$$
where $U=U(Q)$ is the inverse mirror map. On the bottom part of the diagram, the open mirror theorem for compact semi-Fano toric orbifolds (Conjecture \ref{open_mirror_thm}) relates $W^{LF}_\mathcal{X}$ and $W^{HV}_\mathcal{X}$:
$$W^{LF}_\mathcal{X}(q)=W^{HV}_\mathcal{X}(y(q)),$$
where $y=y(q)$ is the inverse of the mirror map $q=q(y)$. One can patch $W^{HV}_Y$ and $W^{HV}_\mathcal{X}$ to form a global family of functions by analyzing the toric data. Open CRC for $\mathcal{X}$ and $Y$ can then be deduced by a suitable analytic continuation of the (inverse) mirror map of $Y$.

\noindent \textbf{Example: $\mathcal{X}=\proj(1,1,2)$.} To illustrate our results, let us consider the $n=2$ case of Theorem \ref{open_CRC_wPn_intro}. Let $N=\Z^2$. The weighted projective plane $\mathcal{X}=\proj(1,1,2)$ is a Gorenstein toric Fano orbifold whose coarse moduli space is the toric variety defined by the simplicial fan $\Sigma_\mathcal{X}$ in $N_\R=\R^2$ generated by
\begin{align*}
\bb_1 = & (1,0), \bb_2 = (-1,2), \bb_3=(0,-1) \in N.
\end{align*}
There is a unique isolated $\Z_2$-singularity at the point corresponding to the cone generated by $\bb_1$ and $\bb_2$.

A crepant resolution of $\mathcal{X}$ is given by the Hirzebruch surface $Y=\mathbb{F}_2$ which is defined by the fan $\Sigma_Y$ in $N_\R$ generated by
\begin{align*}
\bb_1 = & (1,0), \bb_2 = (-1,2), \bb_3=(0,-1), \bb_4=(0,1) \in N.
\end{align*}
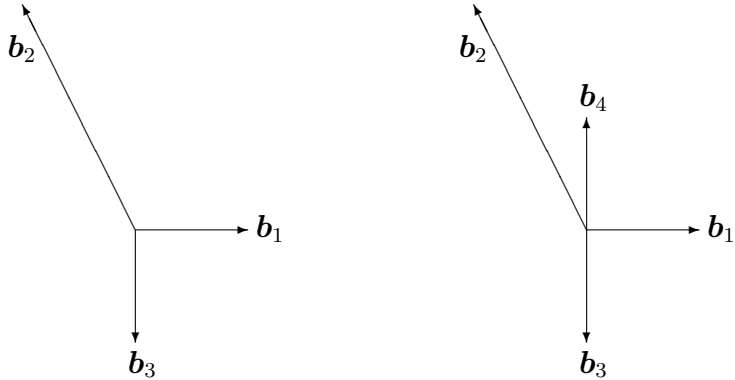
\begin{figure}[ht]
\setlength{\unitlength}{1mm}
\begin{picture}(100,60)
\put(20,18){\vector(1,0){15}} \put(36,17){$\bb_1$} \put(20,18){\vector(0,-1){15}} \put(19,-1){$\bb_3$} \put(20,18){\vector(-1,2){15}} \put(3,41){$\bb_2$}
\put(80,18){\vector(1,0){15}} \put(96,17){$\bb_1$} \put(80,18){\vector(0,-1){15}} \put(79,-1){$\bb_3$} \put(80,18){\vector(-1,2){15}} \put(63,41){$\bb_2$} \put(80,18){\vector(0,1){15}} \put(79,34.5){$\bb_4$}
\end{picture}
\caption{The fans $\Sigma_\mathcal{X}$ (left) and $\Sigma_Y$ (right).}\label{fig1}
\end{figure}

The Lagrangian Floer superpotential $W^{LF}_Y$ was first computed by Auroux \cite{auroux09} using degeneration method and wall-crossing formulas. Different proofs appeared later in \cite{Chan10,FOOO10}. The result is the following
\begin{equation}\label{LF_F2}
W^{LF}_Y(Q_1,Q_2)=z_1+z_2+\frac{Q_1Q_2^2}{z_1z_2^2}+\frac{Q_2(1+Q_1)}{z_2}
\end{equation}
where $z_1,z_2$ are the standard coordinates on $(\C^*)^2$ and $Q_1,Q_2$ are coordinates in the neighborhood $U_Y$ of the cusp corresponding to $Y$ in the stringy K\"ahler moduli space $\CM_A$. $Q_1$ corresponds to the exceptional $(-2)$-curve in $\mathbb{F}_2$ while $Q_2$ corresponds to the fiber class if we view $\mathbb{F}_2$ as a $\proj^1$-bundle over $\proj^1$.

On the other hand, we define the Lagrangian Floer superpotential $W^{LF}_\mathcal{X}$ using counting of Maslov index two smooth holomorphic discs in $\mathcal{X}$ as well as (virtual) counting of orbi-disc with possibly multiple $\tau_2$ orbifold insertions. Here, $\tau_2$ is the orbifold parameter which corresponds to the twisted sector $\mathcal{X}_{1/2}$ supported at the isolated $\Z_2$-singularity. We prove a relation between open and closed orbifold GW invariants (Theorem \ref{THM_open_closed}), and from this we can compute the Lagrangian Floer superpotential $W^{LF}_\mathcal{X}$:
\begin{equation}\label{LF_P(112)}
W^{LF}_\mathcal{X}(q_1,q_2)=z_1+z_2+\frac{q_1}{z_1z_2^2}+\frac{2q_1^{1/2}\sin\frac{\tau_2}{2}}{z_2}
\end{equation}
where $q_1,q_2:=\conste^{\tau_2}$ are coordinates in the neighborhood $U_\mathcal{X}\subset\CM_A$ of the cusp corresponding to $\mathcal{X}$. Here, $q_1$ corresponds to the hyperplane class in $\proj(1,1,2)$.

\begin{figure}[htp]
\begin{center}
\includegraphics[height=2in]{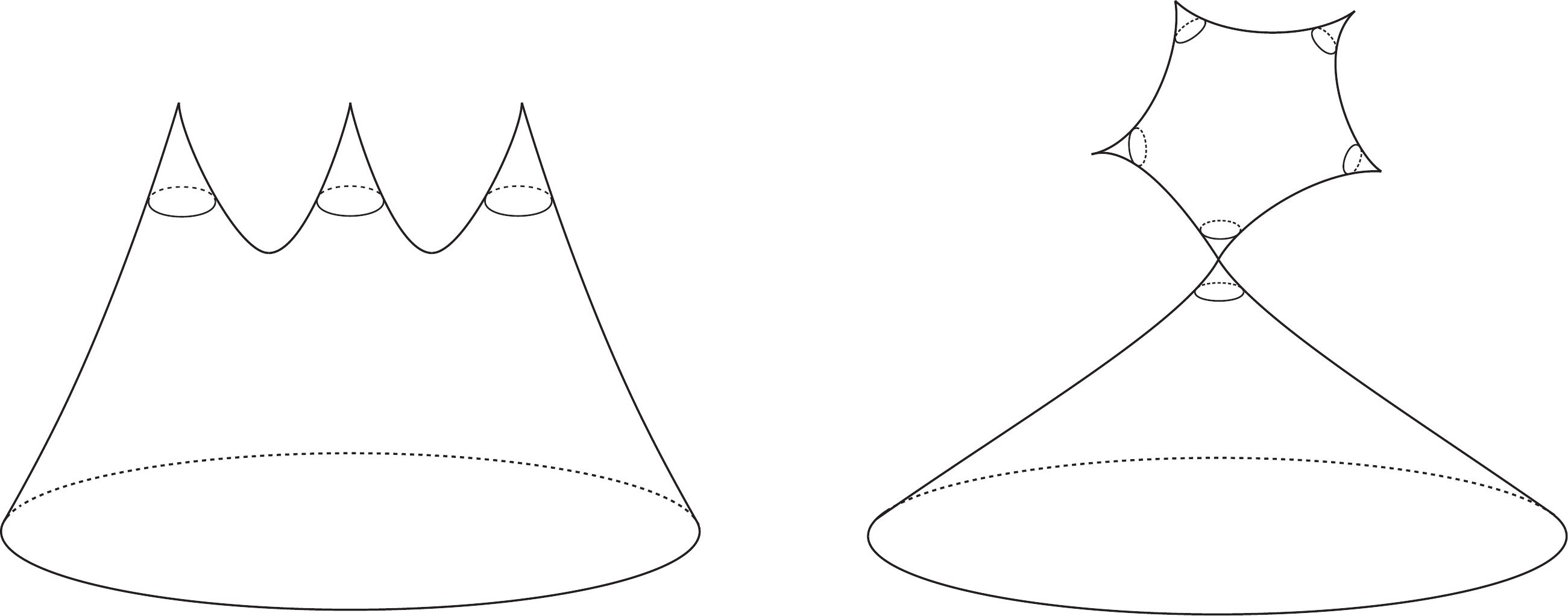}
\caption{An orbi-disc with three orbifold point (left) and a stable orbi-disc (right).}\label{orbidisc}
\end{center}
\end{figure}

The coefficient
$$2\sin\frac{\tau_2}{2}=\tau_2-\frac{\tau_2^3}{3!\cdot 2^2}+\frac{\tau_2^5}{5!\cdot 2^4}-\frac{\tau_2^7}{7!\cdot 2^6}+\cdots$$
is the generating function of the (virtual) counts of stable holomorphic orbi-discs with interior orbifold marked points mapped to the twisted sector $\mathcal{X}_{1/2}$. The first term $\tau_2$ corresponds to the basic orbi-disc classified in \cite{CP}, and the subsequent contributions with multiple $\tau_2$-insertions all come from the {\em same} minimal homotopy class of the basic orbi-disc. Namely the latter is from the virtual perturbation of the orbi-disc attached with constant orbi-sphere bubble as shown in the right-hand-side of Figure \ref{orbidisc}; actual holomorphic orbi-discs with more than one $\tau_2$ insertions do {\em not} have this minimal homotopy class.

A glance at the formulas (\ref{LF_F2}) and (\ref{LF_P(112)}) immediately shows that the substitution
\begin{equation}\label{change_of_variables_P(112)}
Q_1=\conste^{-\mathbf{i}(\pi-\tau_2)},\ Q_2=q_1^{1/2}\conste^{\mathbf{i}(\pi-\tau_2)/2}
\end{equation}
will give the open CRC in this example:
$$W^{LF}_Y(Q_1(q_1,q_2),Q_2(q_1,q_2))=W^{LF}_\mathcal{X}(q_1,q_2).$$
We emphasis that there is an analytic continuation hidden here: a priori the Lagrangian Floer superpotential $W^{LF}_Y(Q_1,Q_2)$ is defined only when the quantum parameters $Q_1,Q_2$ are small, say $|Q_1|,|Q_2|\ll 1$, so we need to analytically continue $W^{LF}_Y(Q_1,Q_2)$ to places where $|Q_1|=1$.

Notice that the change of variables (\ref{change_of_variables_P(112)}) is affine linear. Hence it preserves the canonical flat structures near the cusps. In fact, it was shown in \cite{CIT09} that the Frobenius manifolds defined by the genus 0 Gromov-Witten theory for the orbifold $\proj(1,1,2)$ and its resolution $\mathbb{F}_2$ are isomorphic after analytic continuation of quantum parameters. This is true in general for any toric orbifold with the Hard Lefschetz property.

Now the specialization
$$Q_1=-1,\ Q_2=\mathbf{i}q_1^{1/2},$$
which corresponds to setting the orbi-parameter $\tau_2=0$, gives the isomorphism
$$QH^*(\mathbb{F}_2)\cong QH^*(\proj(1,1,2))$$
asserted by Ruan's CRC (see \cite[Theorem 1.1]{CIT09}). From the point of view of Lagrangian Floer theory, this specialization corresponds to turning off orbifold parameters $\tau_2=0$, or equivalently, the vanishing of the term $\frac{Q_2(1+Q_1)}{z_2}$ in $W^{LF}_Y$ which counts stable discs in $Y$ which meet the exceptional $(-2)$-curve in $Y=\mathbb{F}_2$. This gives a new geometric interpretation of the specialization. \hfill $\square$\\

\begin{remark}
\hfill
\begin{enumerate}
\item
We point out that for $3$-dimensional toric Calabi-Yau geometry, one can consider open Gromov-Witten invariants with respect to Lagrangian submanifolds of Aganagic-Vafa type \cite{ag_va}. The open crepant resolution conjecture in this setting has also been studied; see Cavalieri-Ross \cite{Ca_Ro} for the case $[\C^2/\Z_2] \times \C$ and the recent vast generalization in Brini-Cavalieri-Ross \cite{Br_Ca_Ro}. We remark that the open orbifold GW invariants in \cite{Ca_Ro, Br_Ca_Ro} (and related works) are defined using localization formulas instead of directly by constructing moduli spaces of orbi-discs.

\item
It is expected that open Gromov-Witten theories of an orbifold and its crepant resolution are related even beyond the toric case. In the toric case, the open Gromov-Witten theory is encoded in the superpotential. In more general cases, one must work with more general objects such as the Fukaya category. It is natural to speculate that a relation between open Gromov-Witten theories of an orbifold and its crepant resolution would take the form of an equivalence between (suitable variants of) their Fukaya categories, after analytic continuation.
\end{enumerate}
\end{remark}

The rest of this paper is arranged as follows. In Section \ref{sec:prelim}, we briefly go through the theory of Maslov index for orbifolds (following the recent work of Cho-Shin \cite{CS}) and open orbifold GW theory for toric orbifolds (following Cho-Poddar \cite{CP}). These are prerequisites for defining the Lagrangian Floer superpotentials and hence LG models mirror to toric orbifolds, which we introduce in Section \ref{sec:LG_mirrors}, where we also state the open toric mirror theorem. In Section \ref{sec:open_CRC}, we formulate the open CRC, discuss its relations with the closed CRC and explain a new geometric meaning of the specialization of quantum parameters in Ruan's conjecture. Section \ref{comparison_thm} contains the proof of the equality between open and closed orbifold GW invariants (Theorem \ref{THM_open_closed}). In Section \ref{examples}, by applying the open-closed equality, we establish the open mirror theorem and deduce the open CRC for the weighted projective space $\mathcal{X}=\proj(1,\ldots,1,n)$.

\section*{Acknowledgment}
We are grateful to Conan Leung for encouragement and related collaborations, and for many useful conversations. We thank Yong-Geun Oh and Dongning Wang for related collaborations. We would also like to thank Hansol Hong for drawing Figure \ref{orbidisc}. The work of K. Chan described in this paper was substantially supported by a grant from the Research Grants Council of the Hong Kong Special Administrative Region, China (Project No. CUHK404412). The work of C.-H. Cho was supported by the National Research Foundation of Korea (NRF) grant funded by the Korea Government (MEST) (No. 2012-0000795 and No. 2012R1A1A2003117). The work of S.-C. Lau was supported by Harvard University and Kavli IPMU, and he sincerely thanks Kaoru Ono and Hiroshi Iritani for useful discussions on Lagrangian Floer theory and toric orbifolds. Part of this work was carried out when the authors met at the Kavli Institute for the Physics and Mathematics of the Universe in June 2012. It is a pleasure to thank them for hospitality and support.  We also thank the referee for helpful comments and suggestions.

\section{Preliminaries}\label{sec:prelim}

In this section, we review the Chern-Weil Maslov index for orbifolds introduced by Cho-Shin \cite{CS} and also the classification of holomorphic orbifold discs and the definition of open orbifold Gromov-Witten (GW) invariants for toric orbifolds following Cho-Poddar \cite{CP}.

\subsection{Maslov index}\label{sec:Maslov}
Given a real $2n$-dimensional symplectic vector bundle $\CE$ over a Riemann surface $\Sigma$ and a Lagrangian subbundle $\CL$ over the boundary $\partial\Sigma$, one can associate a Maslov index to the bundle pair $(\CE,\CL)$, which is defined as the rotation number of $\CL$ in a symplectic trivialization $\CE\cong\Sigma\times\R^{2n}$.

To extend the notion of Maslov index to the orbifold setting, the second author and Shin \cite{CS} introduced its Chern-Weil definition as follows. Let $J$ be a compatible complex structure of $\CE$. A unitary connection $\nabla$ of $\CE$ is called $\CL$-orthogonal if $\CL$ is preserved by the parallel transport via $\nabla$ along the boundary $\partial\Sigma$; see \cite[Definition 2.3]{CS} for the precise definition).
\begin{defn}[\cite{CS}, Definition 2.8]\label{def:cw}
The {\em Chern-Weil Maslov index} of the bundle pair $(\CE,\CL)$ is defined by
$$\mu_{CW}(\CE,\CL)=\frac{\mathbf{i}}{\pi}\int_\Sigma{tr(F_\nabla)}$$
where $F_\nabla\in\Omega^2(\Sigma,End(\CE))$ is the curvature induced by an $\CL$-orthogonal connection $\nabla$.
\end{defn}
It was proved in \cite[Section 3]{CS} that the Chern-Weil definition agrees with the usual one.

Now let $\Sigma$ be a bordered orbifold Riemann surface with interior orbifold marked points $z_1^+,\ldots,z_l^+\in\Sigma$ such that the orbifold structure at each marked point $z_i^+$ is given by a branched covering map $z\mapsto z^{m_i}$ for some positive integer $m_i$. For an orbifold vector bundle $\CE$ over $\Sigma$ and a Lagrangian subbundle $\CL\to\partial\Sigma$, the Chern-Weil Maslov index $\mu_{CW}(\CE,\CL)$ of the pair $(\CE,\CL)$ is defined  by taking an $\CL$-orthogonal connection $\nabla$, which is invariant under the local group action, and evaluating the integral in Definition \ref{def:cw} in an orbifold sense (see \cite[Definition 6.4]{CS}). It was shown in \cite[Proposition 6.5]{CS} that the Maslov index $\mu_{CW}(\CE,\CL)$ is independent of both the choice of the orthogonal unitary connection $\nabla$ and the choice of a compatible complex structure.

In \cite{CP}, the second author and Poddar have introduced yet another orbifold Maslov index, the so-called {\em desingularized Maslov index} $\mu^{de}$, defined by the desingularization process introduced by Chen-Ruan \cite{CR01}. Instead of recalling its definition (for which we refer the reader to \cite[Section 3]{CP}), let us recall the following result relating the Chern-Weil and the desingularized Maslov indices:
\begin{prop}[\cite{CS}, Proposition 6.10]
We have
\begin{equation}\label{eq:1}
\mu_{CW}(\CE,\CL) = \mu^{de}(\CE,\CL) + 2 \sum_{i=1}^l \iota(\CE;z_i^+),
\end{equation}
where $\iota(\CE;z_i^+)$ is the degree shifting number associated to the $\integer_{m_i}$-action on $\CE$ at the $i$-th marked point $z_i^+\in\Sigma$.
\end{prop}

\subsection{Toric orbifolds}\label{sec:toricorbifold}
A compact toric manifold is constructed from a complete fan of smooth rational polyhedral cones (see the books \cite{Au, Fu}). Analogously, a compact toric orbifold can be constructed from a combinatorial object called a {\em stacky fan}, which consists of a complete fan $\Sigma$ of simplicial rational polyhedral cones together with the choice of a multiplicity (or equivalently, a choice of lattice vector) for each 1-dimensional cone in $\Sigma$.

Consider the lattice $N= \Z^n$ and its dual lattice $M=Hom_\Z(N,\Z)$. For any $\Z$-module $R$, we denote $N_R = N \otimes_\Z R$, $M_R = M\otimes_\Z R$ and by $\langle\cdot,\cdot\rangle:M_R\times N_R\to R$ the natural pairing. Let $\Sigma$ be a fan of simplicial rational polyhedral cones. We denote by $\Sigma^{(k)}$ the set of all $k$-dimensional cones in $\Sigma$. The minimal lattice generators of 1-dimensional cones $\Sigma^{(1)}$ are labelled as $G(\Sigma):=\{\bv_1,\ldots,\bv_m\}$, where $\bv_j = (v_{j1},\ldots,v_{jn})\in N$. For each $j$, fix a lattice vector $\bb_j=c_j\bv_j\in N$ where $c_j$ is a positive integer. We call $\{\bb_1,\ldots,\bb_m\}$ the stacky vectors, and denote $\bb=(\bb_1,\ldots,\bb_m)$. The data $(\Sigma,\bb)$ is called a stacky fan, and this defines a toric orbifold as follows (for more details, see Borisov-Chen-Smith \cite{BCS}\footnote{Note that the construction in \cite{BCS} is more general since toric Deligne-Mumford stacks considered there can have non-trivial generic stabilizers. We do not need this generality.}).

Recall that a subset $\CP=\{\bv_{i_1},\ldots,\bv_{i_p}\}\subset G(\Sigma)$ is called a {\em primitive collection} if $\{\bv_{i_1},\ldots,\bv_{i_p}\}$ does not span a $p$-dimensional cone in $\Sigma$, while any $k$-element subset of $\CP$ for $0 \leq k < p$,  spans a $k$-dimensional cone in $\Sigma$. For a primitive collection $\CP = \{\bv_{i_1},\ldots, \bv_{i_p}\}$ in $G(\Sigma)$, we denote
$$\BA(\CP) = \{(z_1,\ldots,z_m)\in\C^m \mid z_{i_1}= \cdots=z_{i_p}=0\}.$$
Consider $Z(\Sigma)=\cup_{\CP} \BA(\CP)$, the closed algebraic subset  in $\C^m$, where $\CP$ runs over every primitive collections in $G(\Sigma)$. We define $U(\Sigma) = \C^m\setminus Z(\Sigma).$

Consider the map $\beta:\Z^m\to N$ which sends the basis vectors $e_i$ to $\bb_i$ for $i=1,\ldots,m$. Note that $\beta$ may not be surjective. We obtain the following exact sequences
by tensoring with $\R$ and $\mathbb{C}^*$:
\begin{equation}\label{kexact2}
 0 \to \mathfrak{k} \to \R^m \stackrel{\beta}{\to} N_\R \to 0.
\end{equation}
\begin{equation}\label{kexact4}
 0 \to K_\C \to (\C^*)^m \stackrel{\beta_{\C^*}}{\to} N_{\C^*} \to 0.
\end{equation}
Here the map $\beta_{\C^*}$ is given by
$$\beta_{\C^*}(\lambda_1,\ldots,\lambda_m) = \left(\prod_j \lambda_j^{b_{j1}},\ldots, \prod_j \lambda_j^{b_{jn}}\right).$$
For a complete stacky fan $(\Sigma, \bb)$, the algebraic torus $K_\C$ acts effectively on $U(\Sigma)$ with finite isotropy groups.  Then, the global quotient
\begin{equation}
\mathcal{X}_\Sigma = U(\Sigma)/K_\C
\end{equation}
is called the {\em compact toric orbifold} associated to $(\Sigma, \bb)$.

Consider a $d$-dimensional cone $\sigma$ in $\Sigma$ generated by $\bb_\sigma=(\bb_{i_1}, \ldots, \bb_{i_d})$. Define
$$\mathrm{Box}_{\bb_\sigma} =\left\{\nu \in N \mid \nu=\sum_{k=1}^d t_{k} \bb_{i_k},\ t_k \in [0,1)\cap\rat\right\}.$$
Note that $\mathrm{Box}_{\bb_\sigma}$ is in a one-to-one correspondence with the finite group $G_{\bb_{\sigma}} = N/N_{\bb_{\sigma}}$,
where $N_{\bb_\sigma}$ is the submodule of $N$ generated by lattice vectors $\{ \bb_{i_1}, \ldots, \bb_{i_d}\}$. It is easy to see that if $\tau \prec \sigma$, then  we have $\mathrm{Box}_{\bb_\tau}\subset \mathrm{Box}_{\bb_\sigma}$. Define
$$\mathrm{Box}_{\bb_\sigma}^{\circ} = \mathrm{Box}_{\bb_\sigma} - \bigcup_{\tau \precneqq \sigma} \mathrm{Box}_{\bb_\tau},$$
and
\begin{equation*}
\mathrm{Box} = \bigcup_{\sigma \in \Sigma^{(n)}} \mathrm{Box}_{\bb_\sigma} = \bigsqcup_{\sigma \in \Sigma} \mathrm{Box}_{\bb_\sigma}^{\circ}.
\end{equation*}
We set $\mathrm{Box}'=\mathrm{Box}\setminus\{0\}$. Then $\mathrm{Box}'$ is in a one-to-one correspondence with the {\em twisted sectors}, i.e. non-trivial connected components of the inertia orbifold of $\mathcal{X}_\Sigma$. We refer the readers to \cite{BCS} for more explanations (see also \cite[Section 3.1]{iritani09} for an excellent review on the essential ingredients of toric orbifolds). For $\nu\in\mathrm{Box}$, we denote by $\mathcal{X}_\nu$ the corresponding twisted sector of $\mathcal{X}$. Note that $\mathcal{X}_0=\mathcal{X}$.

Twisted sectors were used by Chen-Ruan \cite{CR04} to define a cohomology theory for orbifolds. For the toric orbifold $\mathcal{X}$, the {\em Chen-Ruan orbifold cohomology} $H^*_\mathrm{orb}(\mathcal{X};\rat)$ is given by
$$H^d_\mathrm{orb}(\mathcal{X};\rat)=\bigoplus_{\nu\in\mathrm{Box}}H^{d-2\iota(\nu)}(\mathcal{X}_\nu;\rat),$$
where $\iota(\nu)$ is the degree shifting number of the twisted sector $\mathcal{X}_\nu$ and the cohomology groups on the right hand side are singular cohomology groups. In \cite{CR04}, Chen and Ruan introduced a product structure which gives $H^*_\mathrm{orb}(\mathcal{X};\rat)$ a Frobenius algebra structure under the {\em orbifold Poincar\'e pairing}.

By a theorem of Delzant, a symplectic toric manifold is completely determined, up to equivariant symplectomorphisms, by its moment polytope. Lerman and Tolman \cite{LT} generalized this to the orbifold case, showing that a symplectic toric orbifold is completely determined by a simple rational convex polytope together with a positive integer attached to each of its facets.

More precisely, let $P$ be a simple rational convex polytope in $M_\R=\R^n$ with $m$ facets $F_1,\ldots,F_m$. Denote by $\bv_j\in N$ ($j=1,\ldots,m$) the inward normal vector to $F_j$ which is the minimal lattice vector. If we label each facet $F_j$ by a positive integer $c_j$ and set $\bb_j=c_j\bv_j$, then the data $(P,\bb)$ is called a {\em labeled polytope}, where we denote $\bb=(\bb_1,\ldots,\bb_m)$. By choosing suitable $\lambda_j \in \R$, the polytope $P$ can be written as
\begin{equation}\label{def:polytope}
P = \bigcap_{j=1}^m \{ u \in M_\R \mid \langle u,\bb_j \rangle \geq \lambda_j \},
\end{equation}
For each stacky vector $\bb_j$, we define the linear functional $\ell_j:M_\R\to\R$ by
\begin{equation}\label{eq:delldef}
\ell_j(u) = \langle u, \bb_j \rangle - \lambda_j,
\end{equation}
Then, we have
$$P = \bigcap_{j=1}^m\{ u \in M_\R \mid \ell_j(u) \geq 0\}.$$
Let $\Sigma(P)$ be the normal fan of $P$. Then the stacky fan $(\Sigma(P),\bb)$ defines a compact toric orbifold $\mathcal{X}_{\Sigma(P)}$ as explained above.

We can now state the theorem of Lerman and Tolman as follows.
\begin{theorem}[\cite{LT}, Theorem 1.5]\label{thm:LeTo}$\mbox{}$
\begin{enumerate}
\item
Let $(\mathcal{X},\omega)$ be a compact symplectic toric orbifold with moment map $\pi:\mathcal{X}\to M_\R$. Then the moment map image $P:=\pi(\mathcal{X})$ is a simple rational convex polytope in $M_\R$, and for each facet $F_j$ of $P$, there exists a positive integer $c_j$ (the label of $F_j$) such that the structure group of every $p \in \pi^{-1}(\mathrm{int}(F_j))$ is $\Z/c_j\Z$.
\item
Two compact symplectic toric orbifolds are equivariantly symplectomorphic (with respect to a fixed torus acting on both orbifolds) if and only if their associated labeled polytopes are isomorphic.
\item
Every labeled polytope arises from a compact symplectic toric orbifold $(\mathcal{X},\omega)$.
\end{enumerate}
\end{theorem}

\subsection{Holomorphic (orbi-)discs}

Let $(\mathcal{X},\omega)$ be a compact K\"ahler toric orbifold of complex dimension $n$, equipped with the standard complex structure $J_0$. Denote by $(P,\bb)$ the associated labeled polytope, where $\bb=(\bb_1,\ldots,\bb_m)$ and $\bb_j=c_j\bv_j$. The polytope $P$ is defined as in \eqref{def:polytope}. We let $D_j$ be the toric prime divisor associated to $\bb_j$.

Let $L \subset\mathcal{X}$ be a Lagrangian torus fiber\footnote{Throughout this paper Lagrangian torus fibers are chosen to be general, i.e.  fibers of the moment map over general points in the interior of the moment polytope.} of the moment map $\pi:\mathcal{X}\to P$, and fix a relative homotopy class $\beta \in \pi_2(\mathcal{X},L) = H_2(\mathcal{X},L;\integer)$. We are interested in holomorphic (orbifold) discs in $\mathcal{X}$ bounded by $L$ and representing the class $\beta$.

Let $(\bD, z_1^+,\ldots,z_l^+)$ be an orbifold disc with interior orbifold marked points $z_1^+,\ldots,z_l^+$. Here $\bD$ is analytically the disc $D^2\subset\C$, together with orbifold data at each marked point $z_i^+$ for $i=1,\ldots,l$. For each $i$, the orbifold data at $z_i^+$ is given by a disc neighborhood of $z_i^+$ which is uniformized by a branched covering map $br:z \to z^{m_i}$ for some positive integer $m_i$. If $m_i=1$, we regard $z_i^+$ as a smooth interior marked point.

An {\em orbifold holomorphic disc} in $\mathcal{X}$ with boundary in $L$ is a continuous map
$$w:(\bD,\partial\bD) \to (\mathcal{X},L)$$
such that for any $z_0 \in \bD$, there is a disc neighborhood of $z_0$ with a branched covering map $br:z \to z^m$, and there is a local chart $(V_{w(z_0)},G_{w(z_0)},\pi_{w(z_0)})$ of $\mathcal{X}$ at $w(z_0)$ and a local holomorphic lifting $\WT{w}_{z_0}$ of $w$ satisfying
$$w \circ br = \pi_{w(z_0)} \circ \WT{w}_{z_0}.$$
We additionally assume that the map $w$ is {\em good} (in the sense of Chen-Ruan \cite{CR01}) and {\em representable}. In particular, for each marked point $z_i^+$, we have an associated {\em injective} homomorphism
\begin{equation}\label{eq:locgphi}
h_i:\Z_{m_i}\to G_{w(z_i^+)}
\end{equation}
between local groups which makes $\WT{w}_{z_i^+}$ equivariant. Denote by $\nu_i\in\mathrm{Box}$ the image of the generator $1$ of $\Z_{m_i}$ under $h_i$ and let $\mathcal{X}_{\nu_i}$ be the twisted sector of $\mathcal{X}$ corresponding to $\nu_i$. Such a map is said to be of {\em type} $\bx := (\mathcal{X}_{\nu_1},\ldots, \mathcal{X}_{\nu_l})$.

We recall the following classification theorem due to the second author and Poddar:
\begin{theorem}[\cite{CP}, Theorem 6.2]\label{classification_orbidiscs}
Let $\mathcal{X}$ be a symplectic toric orbifold corresponding to $(\Sigma(P),\bb)$ and $L$ a Lagrangian torus fiber. Consider a fixed orbit $\WT{L}\subset\C^m \setminus Z(\Sigma)$ of the real $m$-torus $T^m$ which projects to $L$. Suppose $w:(\bD,\partial \bD) \to (\mathcal{X},L)$ is a holomorphic map with orbifold singularities at interior marked points $z_1^+,\ldots,z_l^+ \in \bD$. Then
\begin{enumerate}
\item For each orbifold marked point $z_i^+$, we have a twisted sector $\nu_i=\sum_{j=1}^m t_{ij} \bb_{i_j} \in \mathrm{Box}$, obtained as in \eqref{eq:locgphi}.
\item For an analytic coordinate $z$ on $D^2=|\bD|$, the map $w$ can be lifted to a holomorphic map
$$\WT{w}:(D^2,\partial D^2) \to ((\C^m \setminus Z(\Sigma))/K_\C ,\WT{L}/K_\C \cap T^m),$$
so that the homogeneous coordinate functions (modulo $K_\C$-action) $\WT{w}=(\WT{w}_1,\ldots, \WT{w}_m)$ are given by
\begin{equation}\label{geneq}
\WT{w}_j = a_j \cdot \prod_{s=1}^{d_j} \frac{z-\alpha_{j,s}}{1-\overline{\alpha}_{j,s}z} \prod_{i=1}^{l} \left(\frac{z-z_i^+}{1-\overline{z_i^+}z}\right)^{t_{ij}}
\end{equation}
for  $d_j \in \Z_{\geq 0}$ ($j=1,\ldots,m$) and $\alpha_{j,s}\in\mathrm{int}(D^2)$, $a_j\in \C^*$.
\item The map $w$ whose lift is given as \eqref{geneq} satisfies
 $$\mu_{CW}(w) = \sum_{j=1}^m 2d_j+\sum_{i=1}^l 2\iota(\nu_i),$$
where $\iota_{\nu_i}$ is the degree shifting number associated to the twisted sector $\mathcal{X}_{\nu_i}$.
\end{enumerate}
\end{theorem}

In the above theorem, if we set $l=0$ and $d_j=0$ for all $j$ except for one $j_0$ where $d_{j_0}=1$, then the corresponding holomorphic disc is smooth and intersects the associated toric divisor $D_{j_0}\subset X$ with multiplicity one; its homotopy class is denoted as $\beta_{j_0}$. On the other hand, given $\nu\in\mathrm{Box}'$, if we set $l=1$ and $d_j=0$ for all $j$, then we obtain a holomorphic orbi-disc, whose homotopy class is denoted as $\beta_\nu$.

We have the following lemma from Cho-Poddar \cite{CP}:
\begin{lemma}[\cite{CP}, Lemma 9.1]
For $\mathcal{X}$ and $L$ as above,  the relative homotopy group $\pi_2(\mathcal{X},L)$ is generated by the classes $\beta_j$ for $j=1,\ldots,m$ together with $\beta_\nu$ for $\nu\in\mathrm{Box}'$.
\end{lemma}

We call these generators of $\pi_2(\mathcal{X},L)$ the {\em basic disc classes}. They are the analogue of Maslov index two classes in toric manifolds. Recall that the Maslov index two holomorphic discs in toric manifolds are minimal, in the sense that every non-trivial holomorphic disc bounded by a Lagrangian torus fiber has Maslov index at least two. Also, such discs play a prominent role in the Lagrangian Floer theory of Lagrangian torus fibers in toric manifolds, namely, the Floer cohomology of Lagrangian torus fibers are determined by them. Basic disc classes were used in \cite{CP} to define the leading order bulk orbi-potential, and it can be used to determine Floer homology of torus fibers with suitable bulk deformations.

We recall the classification of basic discs from \cite{CP}:
\begin{corollary}[\cite{CP}, Corollaries 6.3 and 6.4]$\mbox{}$
\begin{enumerate}
\item The smooth holomorphic discs of Maslov index two (modulo $T^n$-action and automorphisms of the domain) are in a one-to-one correspondence with the stacky vectors $\{\bb_1,\ldots,\bb_m\}$.
\item The holomorphic orbi-discs with one interior orbifold marked point and desingularized Maslov index zero (modulo $T^n$-action and automorphisms of the domain) are in a one-to-one correspondence with the twisted sectors $\nu \in \mathrm{Box}'$ of the toric orbifold $\mathcal{X}$.
\end{enumerate}
\end{corollary}

For each $\nu \in \mathrm{Box}'$, we introduce the linear functional $\ell_\nu:M_\R\to\R$,
 defined as
\begin{equation}\label{eq:delnudef}
\ell_\nu(u) = \langle u, \nu \rangle - \lambda_\nu,
\end{equation}
which is analogous to \eqref{eq:delldef} for stacky vectors.
Here, $\lambda_\nu$ is the unique constant which makes $\ell_\nu (\pi(\mathcal{X}_\nu)) \equiv 0$.

Alternatively, \eqref{eq:delnudef} can be defined as follows. For $\nu = \sum_{j=1}^m t_j \bb_j$, we can define $$\ell_\nu = \sum_{j=1}^m t_j \ell_j.$$
In this case, we have $\lambda_\nu = \sum_{j=1}^m t_j \lambda_j$. Thus, for any $u \in P$, $\ell_a(u) \geq 0$ for  $a \in \{1,\ldots, m\} \cup \mathrm{Box}'$. Geometrically, $\ell_a(u)$ is the symplectic area (up to a multiple of $2\pi$) of the basic disc class $\beta_a$ bounded by the Lagrangian torus fiber $L(u)$ over $u \in\mathrm{int}(P)$, where $\mathrm{int}(P)$ denotes the interior of the polytope $P$ (see \cite[Lemma 7.1]{CP}).

\subsection{Moduli spaces of holomorphic (orbi-)discs}

Consider the moduli space $\CM^{main}_{k+1,l}(L,\beta,\bx)$ of good representable stable maps from bordered orbifold Riemann surfaces of genus zero with $k+1$ boundary marked points $z_0,z_1\ldots,z_k$ and $l$ interior (orbifold) marked points $z_1^+,\ldots,z_l^+$ in the homotopy class $\beta$ of type $\bx = (\mathcal{X}_{\nu_1},\ldots, \mathcal{X}_{\nu_l})$. Here, the superscript ``$main$" indicates that we have chosen a connected component on which the boundary marked points respect the cyclic order of $S^1=\partial D^2$. Let $\CM^{main,reg}_{k+1,l}(L,\beta,\bx)$ be its subset consisting of all maps from an (orbi-)disc (i.e. without (orbi-)sphere and (orbi-)disc bubbles). It was shown in \cite{CP} that $\CM^{main}_{k+1,l}(L,\beta,\bx)$ has a Kuranishi structure of real virtual dimension
\begin{equation}\label{eq:dimes}
n + \mu^{de}(\beta,\bx) + k + 1 + 2l - 3 = n + \mu_{CW}(\beta) + k + 1 + 2l -3 - 2\sum_{i=1}^l\iota(\nu_i).
\end{equation}



The following proposition was proved in \cite{CP}.
\begin{prop}[\cite{CP}, Proposition 9.4]\label{prop:moduli}
$\mbox{}$
\begin{enumerate}
\item  Suppose that $\mu^{de}(\beta,\bx)<0$. Then, $\CM^{main,reg}_{k+1,l}(L,\beta,\bx)$ is empty.

\item For $\beta$ satisfying  $\mu^{de}(\beta,\bx)=0$ and  $\beta\neq\beta_\nu$ for any $\nu\in\mathrm{Box}$,  the moduli space $\CM^{main,reg}_{k+1,1}(L,\beta,\bx)$ is empty.

\item For any $\beta$,  $\CM_{k+1,1}^{main,reg}(L,\beta,\bx)$ is Fredholm regular. Moreover, the evaluation map $ev_0:\CM^{main,reg}_{k+1,1}(L,\beta,\bx) \to L$ (at the boundary marked point $z_0$) is a submersion.

\item If $\CM^{main}_{1,1}(L,\beta)$ is non-empty and if $\partial\beta\notin N_{\bb}$, then there exist $\nu\in \mathrm{Box}$, $k_j\in\N$ ($j=1,\ldots,m$) and $\alpha_i\in H_2(X;\Z)$ such that
    $$\beta = \beta_\nu + \sum_{j=1}^m k_j \beta_j + \sum_i \alpha_i,$$
    where each $\alpha_i$ is realized by a holomorphic (orbi-)sphere. 

\item For $\nu\in \mathrm{Box}'$, we have
    $$\CM^{main,reg}_{1,1}(L,\beta_\nu) =\CM^{main}_{1,1}(L,\beta_\nu).$$
    The moduli space $\CM^{main}_{1,1}(L,\beta_\nu)$ is Fredholm regular and the evaluation map $ev_0$ is an orientation preserving diffeomorphism.
\end{enumerate}
\end{prop}

\subsection{Open orbifold Gromov-Witten invariants}

We are now ready to introduce open orbifold GW invariants following Cho-Poddar \cite[Section 12]{CP}.

First of all, fix $l$ twisted sectors $\mathcal{X}_{\nu_1}, \ldots, \mathcal{X}_{\nu_l}$ of the toric orbifold $\mathcal{X}$. Consider the moduli space $\CM^{main}_{1,l}(L,\beta,\bx)$ of good representable stable maps from bordered orbifold Riemann surfaces of genus zero with $1$ boundary marked points and $l$ interior orbifold marked points of type $\bx = (\mathcal{X}_{\nu_1},\ldots, \mathcal{X}_{\nu_l})$ representing the class $\beta$. By \cite[Lemma 12.5]{CP}, for each given $E>0$, there exists a system of multisections $\mathfrak{s}_{\beta,1,l,\bx}$ on $\CM^{main}_{1,l}(L,\beta,\bx)$ for $\beta\cap\omega<E$ which are transversal to 0 and invariant under the $T^n$-action.

Hence, if the virtual dimension of the moduli space is less than $n$, then the perturbed moduli spaces
$\CM^{main}_{1,l}(L,\beta,\bx)^{\mathfrak{s}_{\beta,1,l,\bx}}$ is empty. From the dimension formula \eqref{eq:dimes}, the virtual dimension of the moduli space
$\CM^{main}_{1,l}(L,\beta,\bx)$ is equal to $n$ if and only if
\begin{equation}\label{critmaslov}
\mu_{CW}(\beta) =  2 + \sum_{j=1}^l (2\iota(\nu_i)-2).
\end{equation}

Now let $\beta\in\pi_2(\mathcal{X},L)$ be a class with Maslov index satisfying \eqref{critmaslov}. Then the virtual fundamental chain
$$[\CM_{1,l}(L,\beta,\bx)]^{vir}:=[\CM^{main}_{1,l}(L,\beta,\bx)^{\mathfrak{s}_{\beta,1,l,\bx}}]$$
becomes a {\em cycle} because it has no real codimension one boundaries (because of $T^n$ equivariant perturbation). Hence we can define the following {\em open orbifold GW invariant}:
\begin{defn}
Let $\beta\in\pi_2(\mathcal{X},L)$ be a class with Maslov index $\mu_{CW}(\beta)=2+\sum_{i=1}^l (2\iota(\nu_i)-2)$. Then we define $n_{1,l,\beta}^\mathcal{X}([\mathrm{pt}]_{L};\mathbf{1}_{\nu_1},\ldots,\mathbf{1}_{\nu_l})\in\rat$ by the push-forward
$$n_{1,l,\beta}^\mathcal{X}([\mathrm{pt}]_{L};\mathbf{1}_{\nu_1},\ldots,\mathbf{1}_{\nu_l})=ev_{0*}([\CM_{1,l}(L,\beta,\bx)]^{vir})\in H_n(L;\rat)\cong\rat,$$
where $ev_0:\CM^{main}_{1,l}(L,\beta,\bx)\to L$ is evaluation at the boundary marked point, $[\mathrm{pt}]_{L} \in H^n(L;\rat)$ is the point class of the Lagrangian torus fiber $L$, and $\mathbf{1}_{\nu_i} \in H^0(\mathcal{X}_{\nu_i};\rat)\subset H^{2\iota(\nu_i)}_{\mathrm{orb}}(\mathcal{X};\rat)$ denotes the fundamental class of the twisted sector $\mathcal{X}_{\nu_i}$.
\end{defn}

By \cite[Lemma 12.7]{CP}, the numbers $n_{1,l,\beta}^\mathcal{X}([\mathrm{pt}]_{L};\mathbf{1}_{\nu_1},\ldots,\mathbf{1}_{\nu_l})$ are independent of the choice of the system of multisections used to perturb the moduli spaces, so they are indeed invariants.

Suppose that $\iota(\nu_i)=1$ for all $i$, then $\mu_{CW}(\beta)=2$  satisfies the condition \eqref{critmaslov} for any number of interior orbifold marked points. Thus we can possibly have infinitely many nonzero invariants associated to a given relative homotopy class in this situation. These invariants, as we will see in examples, are quite non-trivial, and it is in sharp contrast with the manifold case. In the case of manifolds, the virtual counting of discs with repeated insertions of interior marked points, which are required to pass through divisors (analogous to $\iota(\nu_i)=1$), are determined by an open analogue of the divisor equation (see \cite{Cho05,FOOO2}). Note that, however, the divisor equation does not hold for interior orbifold marked points.

\begin{remark}
Here we only consider bulk deformations from the fundamental classes of twisted sectors (which is why we use the notation $n_{1,l,\beta}^\mathcal{X}([\mathrm{pt}]_{L};\mathbf{1}_{\nu_1},\ldots,\mathbf{1}_{\nu_l})$). This is because for the purpose of this paper we only need bulk deformations from $H^{\leq2}_\textrm{orb}(\mathcal{X})$.

Consider a cycle $A_{i}$  in $\mathcal{X}_{\nu_i}$, and $\tau_i$ the Poincar\'e dual of $A_{i}$ in $\mathcal{X}_{\nu_i}$. As a cohomology class in $\mathcal{X}_{\nu_i}$, $\tau_i$ is of degree $2d_i:=\textrm{dim}_\R(\mathcal{X}_{\nu_i})-\textrm{dim}_\R(A_{i})$, i.e. $\tau_i\in H^{2d_i}(\mathcal{X}_{\nu_i})\subset H^{2d_i+2\iota(\nu_i)}_\textrm{orb}(\mathcal{X})$. So the condition $\tau_i\in H^{\leq2}_\textrm{orb}(\mathcal{X})$ forces $\tau_i$ to have cohomological degree $0$ in $\mathcal{X}_{\nu_i}$, i.e. we must have $d_i=0$ or $\textrm{dim}_\mathbb{R}(\mathcal{X}_{\nu_i})=\textrm{dim}_\mathbb{R}(A_i)$. This explains why we only consider bulk deformations from fundamental classes of twisted sectors and the invariants $n_{1,l,\beta}^\mathcal{X}([\mathrm{pt}]_{L};\mathbf{1}_{\nu_1},\ldots,\mathbf{1}_{\nu_l})$.
We hope to discuss the general case elsewhere.
\end{remark}

\begin{corollary}\label{openGW_basic}
For $\nu \in \mathrm{Box}'$, we have
$$ n_{1,1,\beta_\nu}^\mathcal{X}([\mathrm{pt}]_{L};\mathbf{1}_\nu) =1.$$
For $j\in \{1,\ldots,m\}$, we have
$$n_{1,0,\beta_j}^\mathcal{X}([\mathrm{pt}]_{L}) =1.$$
\end{corollary}
\begin{proof}
It is not hard to see that the count is one up to sign from the classification theorem. But the sign has been carefully computed in \cite{Cho06} in the toric manifold case, and the orientations in this orbifold case is completely analogous.
\end{proof}

\noindent\textbf{Example: orbifold sphere with two orbifold points.} To illustrate the importance of dimension counting, let us consider an orbifold sphere with two orbifold points with $\Z_p, \Z_q$ singularities.

Let $\mathcal{X}=\proj^1_{p,q}$ for $p, q \in \integer_{>0}$ and a circle fiber $L \subset \proj^1_{p,q}$. There are two orbifold points $x \cong [\{0\}/\integer_p]$ and $x' \cong [\{\infty\}/\integer_q]$. The twisted sectors are given by:
\begin{align*}
& \mathcal{X}_0 = \mathcal{X},\\
& \mathcal{X}_{1/p}, \ldots, \mathcal{X}_{(p-1)/p} \textrm{ supported at $x$, and}\\
& \mathcal{X}_{1/q}, \ldots, \mathcal{X}_{(q-1)/q} \textrm{ supported at $x'$}.
\end{align*}
The total orbifold cohomology ring
$$H_{\mathrm{orb}}^*(\proj^1_{p,q}) = H_{\mathrm{orb}}^0 \oplus H_{\mathrm{orb}}^{2/p} \oplus \ldots \oplus H_{\mathrm{orb}}^{(2p-2)/p} \oplus H_{\mathrm{orb}}^{2/q} \oplus \ldots \oplus H_{\mathrm{orb}}^{(2q-2)/q} \oplus H_{\mathrm{orb}}^2$$
is generated by
$$\mathbf{1}_{\mathcal{X}}, \mathbf{1}_{1/p}, \ldots, \mathbf{1}_{(p-1)/p}, \mathbf{1}'_{1/q}, \ldots, \mathbf{1}'_{(q-1)/q}, [\mathrm{pt}].$$
Here, $\mathbf{1}_{\mathcal{X}}$ and $[\mathrm{pt}]$ have degree shifting numbers equal to zero, while $\mathbf{1}_{i/p}$ and $\mathbf{1}'_{j/q}$ have degree shifting numbers $i/p$ and $j/q$ respectively.

Any disc class $\beta$ is generated by the basic disc classes.  In this case they consist of basic smooth disc classes $\beta_0$ and $\beta'_0$, and basic orbi-disc classes $\beta_{i/p}$ which pass through $\mathcal{X}_{i/p}$ for $i = 1, \ldots, p-1$, and $\beta'_{j/q}$ which pass through $\mathcal{X}_{j/q}$ for $j = 1, \ldots, q-1$.  $\beta_0$ and $\beta'_0$ have Maslov index two, while $\beta_{i/p}$ and $\beta'_{j/q}$ have Maslov index $2i/p$ and $2j/q$ respectively.

Let $\tau_i$ be one of the classes $\mathbf{1}_{1/p}, \ldots, \mathbf{1}_{(p-1)/p}, \mathbf{1}'_{1/q}, \ldots, \mathbf{1}'_{(q-1)/q}$ for $i = 1, \ldots, l$.  By dimension counting, $n_{\beta}([\mathrm{pt}]_{L};\tau_1,\ldots,\tau_l) \neq 0$ only when
$$ \mu_{CW}(\beta) = 2 - \sum_{i=1}^l (2-2\iota_{\nu_i}).$$
Notice that the right hand side is always smaller than or equal to two.

The above equality is satisfied either when $\beta$ is a basic smooth disc class $\beta_0$ or $\beta'_0$, in which case $\mu_{CW}(\beta) = 2$ and $l= 0$, or when $\beta$ is one of the basic orbi-disc class $\beta_{i/p}$ or $\beta'_{j/q}$, in which case $l=1$, $\mu_{CW}(\beta) = 2i/p$ or $2j/q$ and $\tau_1 = \mathbf{1}_{i/p}$ or $\mathbf{1}'_{j/q}$ respectively. For all these basic classes, the open orbifold GW invariants are equal to one.  All other disc classes cannot satisfy the above equality, since the left hand side must increase for other (non-trivial) disc classes, while the right hand side must decrease when the number of interior orbifold marked points increases. \hfill $\square$

\section{LG models as mirrors for toric orbifolds}\label{sec:LG_mirrors}

The Landau-Ginzburg (LG) models mirror to compact toric manifolds have been written down by Hori and Vafa \cite{hori00}.\footnote{The prediction that the mirrors for toric manifolds (or more generally non-Calabi-Yau manifolds) are given by LG models was made even earlier (perhaps implicitly) in the work of Batyrev, Givental and Kontsevich.} Their recipe is combinatorial in nature. In \cite{Cho06}, the second author and Oh gave a geometric construction of the LG mirrors for compact toric Fano manifolds using Lagrangian Floer theory of the moment map fiber tori and counting of holomorphic discs bounded by them. This was later generalized to any compact toric manifolds by the work of Fukaya, Oh, Ohta and Ono \cite{FOOO1,FOOO2,FOOO10b}. In fact the two constructions are related by mirror maps \cite{CLLT11}; this is the statement of the open mirror theorem.

In this section, we shall introduce the LG models which are mirror to compact toric orbifolds and formulate an orbifold version of the open toric mirror theorem.

\subsection{Extended K\"ahler moduli}

Let $(\mathcal{X},\omega)$ be a compact toric K\"ahler orbifold of complex dimension $n$ associated to a labeled polytope $(P,\bb)$, where $\bb=(\bb_1,\ldots,\bb_m)$ and $\bb_j=c_j\bv_j$ denote the stacky vectors. Let $(\Sigma(P),\bb)$ be the corresponding stacky fan.
\begin{defn}\label{defn:semiFano}
A complex orbifold $\mathcal{X}$ is called {\em (semi-)Fano} if for every non-trivial rational (orbi-)curve $\mathcal{C}\subset \mathcal{X}$, $c_1^{CW}(\mathcal{C}) > 0$ ($\geq 0$).
\end{defn}
From now on, we assume that the following conditions are satisfied (cf. Iritani \cite[Remark 3.4]{iritani09}):
\begin{assumption}\label{assumption}$\mbox{}$
\begin{enumerate}
\item[(1)] $\mathcal{X}$ is semi-Fano, and\\
\item[(2)] the set $\{\bb_1,\ldots,\bb_m\} \cup \{\nu\in\textrm{Box}' \mid \iota(\nu)\leq 1\}$ generates the lattice $N$ over $\Z$.
\end{enumerate}
\end{assumption}
In this case, we enumerate the set $\{\nu\in\textrm{Box}' \mid \iota(\nu)\leq 1\}$ as
$$\{\nu\in\textrm{Box}' \mid \iota(\nu)\leq 1\}=\{\bb_{m+1},\ldots,\bb_{m'}\},$$
where each $\bb_j$ ($j=m+1,\ldots,{m'}$) is of the form
$$\bb_j=\nu_j=\sum_{k=1}^m t_{jk}\bb_i\in N,\ t_{jk}\in[0,1)\cap\rat.$$
The stacky fan together with these extra vectors $\bb_{m+1},\ldots,\bb_{m'}$ constitute an {\em extended stacky fan} (in the sense of Jiang \cite{jiang08}).

Consider the map $\beta^e:\Z^{m'}\to N$ sending the basis vectors $e_j$ to $\bb_j$ for $j=1,\ldots,m'$. By (2) of our assumption, this map is surjective. Hence we have the exact sequence
\begin{equation}\label{kexact5}
 0 \to \mathbb{L} \stackrel{\iota}{\to} \Z^{m'} \stackrel{\beta^e}{\to} N \to 0,
\end{equation}
where $\mathbb{L}:=\mathrm{Ker}(\beta^e)$. Let $r':=m'-n$ denote the rank of $\mathbb{L}$ and let $r:=m-n$ denote the rank of $H_2(\mathcal{X};\rat)$ so that $r'=r+(m'-m)$. We choose an integral basis
$$d_a=\sum_{j=1}^{m'} d_{aj}e_j\in\Z^{m'},\ a=1,\ldots,r'$$
of $\mathbb{L}$ such that $d_{aj}=0$ when $1\leq a\leq r$ and $m+1\leq j\leq m'$, and $\{d_1,\ldots,d_r\}$ provides a positive basis of $H_2(\mathcal{X};\rat)$.

Let $\{p_1,\ldots,p_{r'}\}$ be the basis of $\mathbb{L}^\vee$ dual to $\{d_1,\ldots,d_{r'}\}$. Then the images of $p_1,\ldots,p_r$ in $H^2(\mathcal{X};\R)$ is a nef basis $\{\bar{p}_1,\ldots,\bar{p}_r\}$ of $H^2(\mathcal{X};\R)$ and those of $p_{r+1},\ldots,p_{r'}$ are zero. Define elements $D_j\in\mathbb{L}^\vee$ ($j=1,\ldots,m'$) by
$$D_j=\sum_{a=1}^{r'} d_{aj}p_a$$
so that the map $\iota$ in \eqref{kexact5} is given by $\iota=(D_1,\ldots,D_{m'})$. Over the rational numbers, we have (cf. \cite[Section 3.1.2]{iritani09})
\begin{align*}
H_2(\mathcal{X};\rat) & \cong \mathrm{Ker}((D_{m+1},\ldots,D_{m'}):\mathbb{L}\otimes\rat\to\rat^{m'-m})\\
H^2(\mathcal{X};\rat) & \cong \mathbb{L}^\vee\otimes\rat\Big\slash \bigoplus_{j=m+1}^{m'}\rat D_j.
\end{align*}
We can also identify $\mathbb{L}^\vee\otimes\rat$ with the subspace
$$H^2(\mathcal{X})\oplus\bigoplus_{j=m+1}^{m'} H^0(\mathcal{X}_{\bb_j})\subset H^{\leq2}_\mathrm{orb}(\mathcal{X})$$
where $D_j$ is corresponding to $\mathbf{1}_{\nu_j}$ for $j=m+1,\ldots,m'$.

We denote by $\bar{D}_j$ the image of $D_j$ in $H^2(\mathcal{X};\R)$. Note that for $j=1,\ldots,m$, $\bar{D}_j$ is the Poincar\'e dual of the corresponding toric divisor $D_j\subset\mathcal{X}$, i.e.
$$\bar{D}_j=\sum_{a=1}^r d_{aj}\bar{p}_a=\mathrm{PD}(D_j)\in H^2(\mathcal{X};\R);$$
while $\bar{D}_j=0\in H^2(\mathcal{X};\R)$ for $j=m+1,\ldots,m'$.

Now let $K_\mathcal{X}\subset H^2(\mathcal{X};\R)=H^{1,1}(\mathcal{X};\R)$ be the K\"ahler cone of $\mathcal{X}$.
\begin{defn}
The {\em extended K\"ahler cone} of $\mathcal{X}$ is defined by
$$\widetilde{K}_\mathcal{X}:=K_\mathcal{X}\oplus\bigoplus_{j=m+1}^{m'}\R_{>0}D_j\subset\mathbb{L}^\vee\otimes\R.$$
\end{defn}

\subsection{Landau-Ginzburg mirrors}

The mirror of a toric orbifold $\mathcal{X}$ is given by a Landau-Ginzburg (LG) model $(\check{\mathcal{X}},W)$ consisting of a noncompact K\"ahler manifold $\check{\mathcal{X}}$ together with a holomorphic function $W:\check{\mathcal{X}}\to\C$. The manifold $\check{\mathcal{X}}$ is simply given by the bounded domain $\check{\mathcal{X}}:=\mathrm{int}(P)\times M_\R/M$ in the algebraic torus $M_{\C^*}\cong(\C^*)^n$. The holomorphic function $W$, usually called the {\em superpotential} of the LG model, can be constructed in two ways, one is combinatorial and the other is geometric. The open toric mirror theorem says that these two constructions are related by a mirror map.

First of all, Let $\be_1,\ldots,\be_n$ be the standard basis of $N=\Z^n$. Then each $\be_k$ defines a coordinate function
$$z_k:=\exp(2\pi\consti\langle \cdot, \be_k\rangle):M_{\C^*}\to\C.$$
Let $\CM^\mathcal{X}_B:=\mathbb{L}^\vee\otimes\C^*$ be the {\em $B$-model moduli space} for $\mathcal{X}$. The basis $d_1,\ldots,d_{r'}$ of $\mathbb{L}$ defines $\C^*$-valued coordinates $y_1,\ldots,y_{r'}$ on $\CM^\mathcal{X}_B$.
\begin{defn}\label{W^HV}
The {\em extended Hori-Vafa superpotential} of $\mathcal{X}$ is the function $W^{HV}_\mathcal{X}:\check{\mathcal{X}} \to \C$ defined by
\begin{eqnarray*}
W^{HV}_\mathcal{X}=\sum_{j=1}^{m'} C_j z^{\bb_j},
\end{eqnarray*}
where $z^v$ denotes the monomial $z_1^{v^1}\cdots z_n^{v^n}$ if $v=\sum_{k=1}^n v^k\be_k \in N$ and the coefficients $C_j$ are subject to the following constraints
$$y_a=\prod_{j=1}^{m'} C_j^{d_{aj}},\ a=1,\ldots,r'.$$
This defines a family of functions $\{W^{HV}_\mathcal{X}(y)\}$ parametrized by $y=(y_1,\ldots,y_{r'})\in \CM^\mathcal{X}_B$
\end{defn}

On the other hand, by identifying $\mathbb{L}^\vee\otimes\C$ with the subspace
$$H^2(\mathcal{X})\oplus\bigoplus_{j=m+1}^{m'} H^0(\mathcal{X}_{\bb_j})\subset H^{\leq2}_\mathrm{orb}(\mathcal{X}),$$
we will regard $\CM^\mathcal{X}_A:=\mathbb{L}^\vee\otimes\C^*$ also as the {\em $A$-model moduli space} for $\mathcal{X}$. We equip $\CM^\mathcal{X}_A$ with another set of $\C^*$-valued coordinates $q_1,\ldots,q_{r'}$ corresponding to the same basis $d_1,\ldots,d_{r'}\in\mathbb{L}$. Since $\mathbb{L}^\vee\otimes\C=H^2(\mathcal{X})\oplus\bigoplus_{j=m+1}^{m'} H^0(\mathcal{X}_{\bb_j})$, we can write an element $\tau\in\mathbb{L}^\vee\otimes\C$ as $\tau=\tau_{0,2}+\tau_\mathrm{tw}$ where
\begin{align*}
\tau_{0,2} & = \sum_{a=1}^r\tau_a \bar{p}_a\in H^2(\mathcal{X}),\\
\tau_\mathrm{tw} & =\sum_{a=r+1}^{r'}\tau_a\mathbf{1}_{\bb_{m+a-r}}\in\bigoplus_{j=m+1}^{m'} H^0(\mathcal{X}_{\bb_j}).
\end{align*}
This defines the coordinates $q_a=\exp(\tau_a)$ for $a=1,\ldots,r$ and $\tau_a$ for $a=r+1,\ldots,r'$ on $\CM^\mathcal{X}_A$. Note that the coordinates in the orbifold directions are {\em not} exponentiated coordinates.

We can now define a LG superpotential using Lagrangian Floer theory in terms of the open orbifold GW invariants $n_{1,l,\beta}^\mathcal{X}([\mathrm{pt}]_{L};\mathbf{1}_{\nu_1},\ldots,\mathbf{1}_{\nu_l})$. For $u\in\mathrm{int}(P)$, let $L:=L(u)\subset \mathcal{X}$ be the corresponding Lagrangian torus fiber of the moment map.
\begin{defn}\label{W^LF}
The {\em Lagrangian Floer superpotential} of $\mathcal{X}$ is the function $W^{LF}_\mathcal{X}:\check{\mathcal{X}} \to \C$ defined by
\begin{align*}
W^{LF}_\mathcal{X} & = \sum_{\beta\in\pi_2(\mathcal{X},L)}\sum_{l\geq0}
\frac{1}{l!}n_{1,l,\beta}^\mathcal{X}([\mathrm{pt}]_{L};\tau_\mathrm{tw},\ldots,\tau_\mathrm{tw})Z_\beta\\
& =\sum_{\beta\in\pi_2(\mathcal{X},L)}\sum_{l\geq0}\sum_{a_1,\ldots,a_l}
\frac{\tau_{a_1}\cdots\tau_{a_l}}{l!}n_{1,l,\beta}^\mathcal{X}([\mathrm{pt}]_{L};\mathbf{1}_{\nu_{m+a_1-r}},\ldots,\mathbf{1}_{\nu_{m+a_l-r}})Z_\beta,
\end{align*}
where $Z_\beta$ is the monomial given by
\begin{eqnarray*}
Z_\beta(u,\theta)=\exp\left(-\int_\beta\omega+2\pi\consti\langle \partial\beta,\theta\rangle\right),
\end{eqnarray*}
the third summation is over all $a_1,\ldots,a_l\in\{r+1,\ldots,r'\}$. The superscript ``LF" refers to ``Lagrangian Floer".
\end{defn}
Here, if $\beta=\sum_{j=1}^m k_j\beta_j+\sum_{j=m+1}^{m'}k_j\beta_{\nu_j}+d$ where $d\in H_2^\mathrm{eff}(\mathcal{X})$, then $\partial\beta=\sum_{j=1}^m k_j\bb_j+\sum_{\nu\in\mathrm{Box}'}k_\nu\nu\in N$, so that $Z_\beta=q^d\prod_{j=1}^{m'}Z_j^{k_j}$, where $q^d=q_1^{\langle\bar{p}_1,d\rangle}\cdots q_r^{\langle\bar{p}_r,d\rangle}$ and $Z_j=C_jz^{\bb_j}$ for $j=1,\ldots,m'$, are monomials such that the coefficients $C_j$ are subject to the following constraints:
\begin{enumerate}
\item[(1)] For $a=1,\ldots,r$, the element $d_a\in\mathbb{L}$ is a class in $H_2(\mathcal{X};\rat)$ and the constraint is given by
    $$q_a=\prod_{j=1}^{m'} C_j^{d_{aj}}=\prod_{j=1}^m C_j^{d_{aj}}.$$
\item[(2)] For $a=r+1,\ldots,r'$, the element $d_a\in\mathbb{L}$ corresponds to the relation $\sum_{j=1}^{m'}d_{aj}\bb_j=0$. For $j=m+1,\ldots,m'$, write $\bb_j=\sum_{k=1}^m t_{jk}\bb_k$. Then the previous relation can be rewritten as
    $$\sum_{j=1}^m\left(d_{aj}+\sum_{k=m+1}^{m'}d_{ak}t_{kj}\right)\bb_j=0.$$
    This corresponds to a class $\tilde{d}_a\in H_2(\mathcal{X};\rat)$, and the constraint is given by
    $$q^{\tilde{d}_a}=\prod_{j=1}^{m'} C_j^{d_{aj}}.$$
\end{enumerate}
We emphasize that the coefficients $C_j$'s only depend on the exponentiated coordinates $q_1,\ldots,q_r$, but {\em not} the orbifold parameters $\tau_{r+1},\ldots,\tau_{r'}$. Also note that we need to choose the branches of fractional powers of $q_a$ for $a=1,\ldots,r$ due to the orbifold structure near the cusp in $\CM^\mathcal{X}_A$. Altogether this defines a family of functions $\{W^{LF}_\mathcal{X}(q)\}$ parametrized by $q=(q_1,\ldots,q_r,\tau_{r+1},\ldots,\tau_{r'})\in \CM^\mathcal{X}_A$.

Throughout this paper, we assume that the infinite sum on the right hand side of the above definition converges. Strictly speaking, the above just defines a $\Lambda_0$-valued function where $\Lambda_0$ is the Novikov ring. Assuming convergence, then both $W^{HV}_\mathcal{X}$ and $W^{LF}_\mathcal{X}$ are holomorphic functions on $\check{\mathcal{X}}$ and can be analytically continued to the whole $M_{\C^*}$.

For each $\beta\in\pi_2(\mathcal{X},L)$, the coefficient of $Z_\beta$ is the generating function
$$\sum_{l\geq0}\sum_{a_1,\ldots,a_l}
\frac{1}{l!}n_{1,l,\beta}^\mathcal{X}([\mathrm{pt}]_{L};\mathbf{1}_{\nu_{m+a_1-r}},\ldots,\mathbf{1}_{\nu_{m+a_l-r}})\tau_{a_1}\cdots\tau_{a_l}$$
of open orbifold GW invariants. When $l=0$, $n_{1,0,\beta}([\mathrm{pt}]_{L})$ counts the virtual number of stable smooth holomorphic discs representing $\beta$; when $l=1$,  $n_{1,1,\beta}([\mathrm{pt}]_{L};\mathbf{1}_\nu)$ counts the virtual number of stable holomorphic orbi-discs with one interior orbifold marked point mapping to the twisted sector $\mathcal{X}_\nu$ representing $\beta$.

Following \cite{CP}, we define the {\em leading order superpotential} $W^{LF}_{\mathcal{X},0}$ to be
\begin{equation*}
W^{LF}_{\mathcal{X},0}:=\sum_{j=1}^m n_{1,0,\beta_j}^\mathcal{X}([\mathrm{pt}]_{L})Z_{\beta_j}+\sum_{j=m+1}^{m'} n_{1,1,\beta_{\nu_j}}^\mathcal{X}([\mathrm{pt}]_{L};\mathbf{1}_{\nu_j})\tau_{r+j-m}Z_{\beta_{\nu_j}};
\end{equation*}
By Corollary \ref{openGW_basic}, $W^{LF}_{\mathcal{X},0}$ can be written as
\begin{align*}
W^{LF}_{\mathcal{X},0} = \sum_{j=1}^m Z_{\beta_j}+\sum_{j=m+1}^{m'}\tau_{r+j-m}Z_{\beta_{\nu_j}} = \sum_{j=1}^m C_j z^{\bb_j}+\sum_{j=m+1}^{m'} C_j \tau_{r+j-m}z^{\bb_j},
\end{align*}
where the coefficients $C_j$ ($j=1,\ldots,m'$) are subject to the following constraints
$$q^{\tilde{d}_a}=\prod_{j=1}^{m'} C_j^{d_{aj}},\ a=1,\ldots,r'.$$
Note that the terms in the extended Hori-Vafa superpotential $W^{HV}_\mathcal{X}$ are in a one-to-one correspondence with those in $W^{LF}_{\mathcal{X},0}$. In view of this, we may regard both $W^{HV}_\mathcal{X}$ and $W^{LF}_{\mathcal{X},0}$ as counting only the basic holomorphic (orbi-)discs. The remaining higher order terms in $W^{LF}_\mathcal{X}$ are instanton corrections coming from virtual counting of non-basic holomorphic orbi-discs. The open mirror theorem below asserts that these are precisely the correction terms that we get when we plug in the mirror map into $W^{HV}_\mathcal{X}$.

\subsection{An open mirror theorem}\label{sec:open_mirror_thm}

Let us first recall the mirror theorem for toric orbifolds following Iritani \cite{iritani09}. Consider the subsets
\begin{align*}
\mathbb{K} & =\{d\in\mathbb{L}\otimes\rat\mid\{j\in\{1,\ldots,m'\}\mid\langle D_j,d\rangle\in\Z\}\in\mathcal{A}\},\\
\mathbb{K}_\mathrm{eff} & =\{d\in\mathbb{L}\otimes\rat\mid\{j\in\{1,\ldots,m'\}\mid\langle D_j,d\rangle\in\Z_{\geq0}\}\in\mathcal{A}\},
\end{align*}
where $\mathcal{A}$ is the set of so-called ``anticones". Basically, $\mathbb{K}_\mathrm{eff}$ is the set of effective classes; we refer the reader to \cite[Section 3.1]{iritani09} for the precise definitions. For any real number $r\in\R$, we denote by $\lceil r \rceil$, $\lfloor r \rfloor$ and $\{r\}$ the ceiling, floor and fractional part of $r$ respectively. Then for $d\in\mathbb{K}$, we define
$$\nu(d):=\sum_{j=1}^{m'}\lceil\langle D_j,d\rangle\rceil\bb_j\in N.$$
Notice that we can write
$$\nu(d)=\sum_{j=1}^{m'}(\{-\langle D_j,d\rangle\}+\langle D_j,d\rangle)\bb_j=\sum_{j=1}^{m'}\{-\langle D_j,d\rangle\}\bb_j,$$
so $\nu(d)\in\mathrm{Box}$ and hence it corresponds to a twisted sector $\mathcal{X}_{\nu(d)}$ of $\mathcal{X}$.
\begin{defn}
The {\em $I$-function} of a toric orbifold $\mathcal{X}$ is an $H^*_\mathrm{orb}(\mathcal{X})$-valued power series on $\CM^\mathcal{X}_B$ defined by
\begin{align*}
I_\mathcal{X}(y,z)= & \conste^{\sum_{a=1}^r\bar{p}_a\log y_a/z}\\
& \quad\left(\sum_{d\in\mathbb{K}_\mathrm{eff}}y^d\frac{\prod_{j:\langle D_j,d\rangle<0}\prod_{k\in[\langle D_j,d\rangle,0)\cap\Z}(\bar{D}_j+(\langle D_j,d\rangle-k)z)}{\prod_{j:\langle D_j,d\rangle>0}\prod_{k\in[0,\langle D_j,d\rangle)\cap\Z}(\bar{D}_j+(\langle D_j,d\rangle-k)z)}\mathbf{1}_{\nu(d)}\right),
\end{align*}
where $y^d=y_1^{\langle p_1,d\rangle}\cdots y_{r'}^{\langle p_{r'},d\rangle}$ and $\mathbf{1}_{\nu(d)}\in H^0(\mathcal{X}_{\nu(d)})\subset H^{2\iota(\nu(d))}_\mathrm{orb}(\mathcal{X})$ is the fundamental class of the twisted sector $\mathcal{X}_{\nu(d)}$.
\end{defn}
Under Assumption \ref{assumption}, the $I$-function is a convergent power series in $y_1,\ldots,y_{r'}$ by \cite[Lemma 4.2]{iritani09}. Moreover, it can be expanded as
$$I_\mathcal{X}(y,z)=1+\frac{\tau(y)}{z}+O(z^{-2}),$$
where $\tau(y)$ is a (multi-valued) function with values in $H^{\leq2}_\mathrm{orb}(\mathcal{X})$. We call $q=\exp \tau(y)$ the {\em mirror map}. It defines a local isomorphism near $y=0$ (\cite[Section 4.1]{iritani09}).

On the other hand, we have the following
\begin{defn}
The {\em (small) $J$-function} of a toric orbifold $\mathcal{X}$ is an $H^*_\mathrm{orb}(\mathcal{X})$-valued power series on $\CM^\mathcal{X}_A$ defined by
\begin{align*}
J_\mathcal{X}(q,z)= & \conste^{\tau_{0,2}/z}\left(1+\sum_\alpha\sum_{\stackrel{(d,l)\neq(0,0)}{d\in H_2^\mathrm{eff}(\mathcal{X})}}\frac{q^d}{l!}\left\langle 1,\tau_\mathrm{tw},\ldots,\tau_\mathrm{tw},\frac{\phi_\alpha}{z-\psi}\right\rangle^\mathcal{X}_{0,l+2,d} \phi^\alpha\right),
\end{align*}
where $\log q=\tau=\tau_{0,2}+\tau_\mathrm{tw}\in H^2_\mathrm{orb}(\mathcal{X})$ with $\tau_{0,2}=\sum_{a=1}^r \bar{p}_a\log q_a\in H^2(\mathcal{X})$ and $\tau_\mathrm{tw}=\sum_{a=r+1}^{r'}\tau_a\mathbf{1}_{\bb_{m+a-r}}\in\bigoplus_{j=m+1}^{m'} H^0(\mathcal{X}_{\bb_j})$, $q^d=\conste^{\langle \tau_{0,2},d\rangle}=q_1^{\langle\bar{p}_1,d\rangle}\cdots q_r^{\langle\bar{p}_r,d\rangle}$, $\{\phi_\alpha\}$, $\{\phi^\alpha\}$ are dual basis of $H^*_\mathrm{orb}(\mathcal{X})$ and $\langle\cdots\rangle^\mathcal{X}_{0,l+2,d}$ denote closed orbifold GW invariants.
\end{defn}

Now the mirror theorem  for the toric orbifold $\mathcal{X}$ states that the $J$-function can be obtained from the $I$-function via the mirror map.
This has been recently proved by Coates, Corti, Iritani and the fourth author \cite[Theorem 36]{CCIT13}; see also the formulation in \cite[Section 4.1]{iritani09}.
\begin{theorem}[Closed Mirror Theorem]\label{closed_mirror_thm}
Let $\mathcal{X}$ be a compact toric K\"ahler orbifold satisfying Assumption \ref{assumption}. Then we have
\begin{equation*}
J_\mathcal{X}(q,z)=I_\mathcal{X}(y(q),z),
\end{equation*}
where $y=y(q)$ is the inverse of the mirror map $q=q(y)$.
\end{theorem}

In terms of the extended Hori-Vafa and Lagrangian Floer superpotentials, we suggest the following open string version of the toric mirror theorem:
\begin{conjecture}[Open Mirror Theorem]\label{open_mirror_thm}
Let $\mathcal{X}$ be a compact toric K\"ahler orbifold satisfying Assumption \ref{assumption}, and let $W^{HV}_\mathcal{X}(y)$ and $W^{LF}_\mathcal{X}(q)$ be the extended Hori-Vafa and Lagrangian Floer superpotentials respectively. Then, up to a change of coordinates on $M_{\C^*}$, we have
\begin{equation*}
W^{LF}_\mathcal{X}(q)=W^{HV}_\mathcal{X}(y(q)),
\end{equation*}
where $y=y(q)$ is the inverse of the mirror map $q=q(y)$.
\end{conjecture}
This is the orbifold version of the open toric mirror theorem conjectured by Chan-Lau-Leung-Tseng \cite{CLLT11}. Since $W^{HV}_\mathcal{X}$ and the mirror map are combinatorially defined and can be written down explicitly, the open toric mirror theorem can be used to compute the open orbifold GW invariants $n_{1,l,\beta}^\mathcal{X}([\mathrm{pt}]_{L};\mathbf{1}_{\nu_1},\ldots,\mathbf{1}_{\nu_l})$. In Section \ref{examples}, we will prove Conjecture \ref{open_mirror_thm} for the weighted projective spaces $\mathcal{X}=\proj(1,\ldots,1,n)$ using a formula (Theorem \ref{THM_open_closed}) which equates open and closed orbifold GW invariants for Gorenstein toric Fano orbifolds.

\section{An open crepant resolution conjecture}\label{sec:open_CRC}

In this section, we shall formulate an open string version of the crepant resolution conjecture for toric orbifolds, which says that the Lagrangian Floer superpotentials for a Gorenstein toric orbifold $\mathcal{X}$ and a toric crepant resolution $Y$ coincide after analytic continuation of the Lagrangian Floer superpotential for $Y$ and a suitable change of variables.

\subsection{Formulation of the conjecture}\label{sec:formulation_openCRC}

\begin{defn}
An orbifold $\mathcal{X}$ is called {\em Gorenstein} if its canonical divisor $K_\mathcal{X}$ is Cartier.
\end{defn}

\begin{lemma}
If $\mathcal{X}$ is Gorenstein toric orbifold, then $\mu(\beta) \geq 2$ for any basic disc class $\beta\in \pi_2(\mathcal{X},L)$.
\end{lemma}
\begin{proof}
When $\beta$ is a basic disc class represented by a smooth holomorphic disc, its Maslov index $\mu_{CW}(\beta) = 2$ as in the case of smooth toric manifolds (see  \cite{Cho06} or \cite{CP}).

Consider a basic disc class $\beta_\nu$ for some $\nu \in \mathrm{Box}'$, which is represented by a holomorphic orbi-disc. It was proved in \cite{CP} that $\mu^{de}(\beta_\nu,\mathcal{X}_\nu)=0$. Hence, from Formula \eqref{eq:1}, we have
$$\mu_{CW}(\beta_\nu) = \mu^{de}(\beta_\nu,\mathcal{X}_\nu) + 2 \iota(\nu) = 2 \iota(\nu) \geq 2$$
since being Gorenstein implies that $\iota(\nu)$ is a positive integer.
\end{proof}

Now let $X=X_\Sigma$ be a projective Gorenstein toric variety of complex dimension $n$ defined by a complete fan $\Sigma$ of simplicial rational polyhedral cones. Then $X$ has at worse quotient singularities, and there is a canonical Gorenstein toric orbifold $\mathcal{X}=\mathcal{X}_\Sigma$ with coarse moduli space being $X$ and orbifold structures occurring in complex codimension at least two. We assume that $\mathcal{X}$ satisfies Assumption \ref{assumption}. Note that the canonical toric orbifold $\mathcal{X}$ is the one described by the canonical stacky fan in Borisov-Chen-Smith \cite[Section 7]{BCS}.

\begin{remark}
Let us explain the choice of this canonical $\mathcal{X}$. Given $X$, there can be (infinitely) many orbifolds with coarse moduli space $X$. However, if $\mathcal{X}$ is the canonical toric orbifold associated to $X$ and $p:\mathcal{X}\to X$ is the coarse moduli space map, then $K_\mathcal{X}=p^*K_X$ since p is an isomorphism in complex codimension one. We can understand this as saying that the canonical orbifold $\mathcal{X}$ is a crepant resolution of $X$ since $\mathcal{X}$ is a smooth orbifold, and $p$ is birational and crepant.

On the other hand, all other toric orbifolds with coarse moduli space $X$ are obtained from this canonical $\mathcal{X}$ by root constructions along toric divisors (see \cite{FMN10}). A root construction along a toric divisor introduces orbifoldness along the divisor and changes the canonical divisor by a multiple of that divisor. For example, if $\mathcal{X}'$ is a toric orbifold obtained from the canonical $\mathcal{X}$ by a $r$-th root construction along the toric divisor $D$. Then $X$ is also the coarse moduli space of $\mathcal{X}'$ and the coarse moduli space map $p': \mathcal{X}'\to X$ is birational, but $p'$ is {\em not} crepant since $K_{\mathcal{X}'}=p'^*(K_X+(r-1)/rD)$. So we cannot consider $\mathcal{X}'$ as a crepant resolution of $X$ and consequently $\mathcal{X}'$ is not suited for CRC.
\end{remark}

Let $\pi:Y\to X$ be a toric crepant resolution. Notice that since $\mathcal{X}$ is semi-Fano, so is $Y$, i.e. $c_1(\alpha)\geq0$ for any effective curve class $\alpha\in H_2(Y;\Z)$. Let $\Sigma_Y$ be the fan in $N_\R$ defining $Y$. Then the set of primitive generators of the rays in $\Sigma_Y$ is given by $\Sigma_Y^{(1)}=\{\bb_1,\ldots,\bb_{m'}\}$. Let $\{\alpha_1,\ldots,\alpha_{r'}\}$ be a positive basis for $H_2(Y;\Z)$ such that the classes $\{\pi_*\alpha_1,\ldots,\pi_*\alpha_r\}$ gives precisely the positive basis $\{d_1,\ldots,d_r\}$ for $H_2(\mathcal{X};\Z)$, where $\pi_*:H_*(Y;\rat)\to H_*(\mathcal{X};\rat)$ is the natural push-forward map which is surjective. By identifying $H_2(Y;\Z)$ with $\mathbb{L}$, we can indeed choose
$$\alpha_a=d_a=\sum_{j=1}^{m'} d_{aj}e_j$$
for $a=1,\ldots,r'$. Let $\CM^Y_A:=H^2(Y;\C^*)\cong \mathbb{L}^\vee\otimes\C^*$ be the $A$-model moduli space for $Y$. Then the basis $d_1,\ldots,d_{r'}$ defines $\C^*$-valued coordinates $Q_1,\ldots,Q_{r'}$ on $\CM^Y_A$.

We now fix a choice of a Lagrangian torus fiber $L\subset X$. Since $\pi$ is $T^n$-equivariant, the pre-image of $L\subset X$ is a Lagrangian torus fiber in $Y$, which, by abuse of notations, will again be denoted by $L$. Recall that the relative homotopy group $\pi_2(Y,L)\cong H_2(Y,L;\Z)$ is generated by the basic disc classes $\beta_1,\ldots,\beta_{m'}$, each of which is of Maslov index two and represented by a holomorphic disc $w:(D^2,\partial D^2)\to (Y,L)$. For each $j$, the basic disc class $\beta_j$ intersects with multiplicity one the toric divisor $H_j\subset Y$ which corresponds to the primitive generator $\bb_j$ of a 1-dimensional cone of the fan $\Sigma_Y$. We can identify $\Z^{m'}$ and $N$ with $\pi_2(Y,L)$ and $\pi_1(L)$ respectively so that the exact sequence \eqref{kexact5} becomes
\begin{equation*}
 0 \to H_2(Y;\Z) \to \pi_2(Y,L) \to \pi_1(L) \to 0.
\end{equation*}

Now we recall the definition of open GW invariants of $Y$. Let $\beta\in\pi_2(Y,L)$ be a relative homotopy class with Maslov index $\mu(\beta)=2$. Since $Y$ is semi-Fano, any such $\beta$ is of the form $\beta_j+\alpha$ where $\beta_j$ is a basic disc class and $\alpha\in H_2^\mathrm{eff}(Y;\Z)$ is an effective class represented by holomorphic spheres such that $c_1(\alpha)=0$. Let $\CM^{main}_1(L,\beta)$ the moduli space of stable maps from genus zero bordered Riemann surfaces with one boundary marked point representing the class $\beta$. Then the results of Fukaya-Oh-Ohta-Ono \cite{FOOO1} tell us that $\CM^{main}_1(L,\beta)$ admits a Kuranishi structure of real virtual dimension $n$ and has a virtual fundamental cycle $[\CM_1(L,\beta)]^{vir}$. Define the open GW invariant
$$n_\beta^Y=n_{1,0,\beta}^Y([\mathrm{pt}]_{L})=ev_{0*}([\CM_1(L,\beta)]^{vir})\in H_n(L;\rat)\cong\rat,$$
where $ev_0:\CM^{main}_1(L,\beta)\to L$ is evaluation on the boundary marked point and $[\mathrm{pt}]_{L}\in H^0(L;\rat)$ is the point class of the Lagrangian torus fiber.
\begin{defn}
The Lagrangian Floer superpotential of $Y$ is the function $W^{LF}_Y:M_{\C^*} \to \C$ defined by
\begin{align*}
W^{LF}_Y  = \sum_{\stackrel{\beta\in\pi_2(Y,L)}{\mu(\beta)=2}}n_\beta^Y Z_\beta
= \sum_{j=1}^{m'}\left(\sum_{\stackrel{\alpha\in H_2^\mathrm{eff}(Y)}{c_1(\alpha)=0}}n^Y_{\beta_j+\alpha}Q^\alpha\right) Z_j,
\end{align*}
where $Z_\beta$ is the monomial given by
\begin{eqnarray*}
Z_\beta(u,\theta)=\exp\left(-\int_\beta\omega+2\pi\consti\langle \partial\beta,\theta\rangle\right),
\end{eqnarray*}
and $Z_j$ are monomials $C_jz^{\bb_j}$ such that the coefficients $C_j$ are subject to constraints
$$Q_a=\prod_{j=1}^{m'} C_j^{d_{aj}},\ a=1,\ldots,r'.$$
So this defines a family of functions $W^{LF}_Y(Q)$ parametrized by $Q=(Q_1,\ldots,Q_{r'})\in\CM^Y_A$.
\end{defn}
Again we assume that the infinite sum in the above definition converges and $W^{LF}_Y$ defines an analytic function on $M_{\C^*}$.

On the other hand, let $\CM^Y_B:=\mathbb{L}^\vee\otimes\C^*$ be the B-model moduli space for $Y$. The same basis $d_1,\ldots,d_{r'}$ of $\mathbb{L}$ defines another set of $\C^*$-valued coordinates $U_1,\ldots,U_{r'}$ on $\CM^Y_B$.
\begin{defn}\label{W^{HV}_Y}
The Hori-Vafa superpotential of $Y$ is the function $W^{HV}_Y:M_{\C^*} \to \C$ defined by
\begin{eqnarray*}
W^{HV}_Y=\sum_{j=1}^{m'} C_j z^{\bb_j},
\end{eqnarray*}
where the coefficients $C_j$ are subject to the following constraints
$$U_a=\prod_{j=1}^{m'} C_j^{d_{ja}},\ a=1,\ldots,r'.$$
This defines a family of functions $\{W^{HV}_Y(U)\}$ parametrized by $U=(U_1,\ldots,U_{r'})\in \CM^Y_B$.
\end{defn}

In \cite{CLLT11}, the following open mirror theorem for semi-Fano toric manifolds was proposed:
\begin{theorem}\label{open_mirror_thm_Y}
Let $Y$ be a semi-Fano toric manifold. Then, up to a change of coordinates on $M_{\C^*}$, we have
\begin{equation*}
W^{LF}_Y(Q)=W^{HV}_Y(U(Q)),
\end{equation*}
where $U=U(Q)$ is the inverse mirror map.
\end{theorem}
The mirror map $\log Q=\log Q(U)$ for $Y$ is an $H^2(Y)$-valued function given by the $1/z$-coefficient of the $I$-function for $Y$. It defines a local isomorphism near $U=0$, and $U=U(Q)$ is its inverse. Under the assumption that $W^{LF}_Y$ converges, Theorem \ref{open_mirror_thm_Y} was proved in \cite{CLLT11} for all semi-Fano toric manifolds. In a very recent work \cite{CLLT12}, Theorem \ref{open_mirror_thm_Y} was proved for all semi-Fano toric manifolds {\em without} any convergence assumption. Indeed, the convergence of the coefficients of $W^{LF}_Y$ was deduced as a consequence of the main results in \cite{CLLT12}. The proof in \cite{CLLT12}, which uses {\em Seidel spaces}, is much more geometric in nature and is completely different from the analytic proof in \cite{CLLT11}.

We can now formulate the open crepant resolution conjecture (CRC) as follows:

\begin{conjecture}[Open CRC]\label{open_CRC}
Let $\mathcal{X}$ and $Y$ be as above (with the semi-Fano condition).  Let $l$ be the dimension of the K\"ahler moduli of $\mathcal{X}$ (which equals to that of $Y$).  Then there exists
\begin{enumerate}
\item $\epsilon>0$;
\item a coordinate change $Q(q)$, which is a holomorphic map $(\Delta(\epsilon)-\real_{\leq 0})^{l} \to (\cpx^\times)^l$, and $\Delta(\epsilon)$ is an open disc of radius $\epsilon$ in the complex plane;
\item a choice of analytic continuation of coefficients of the Laurent polynomial $W^{LF}_Y(Q)$ to the target of the holomorphic map $Q(q)$,
\end{enumerate}
such that $W^{LF}_Y(Q(q))$ defines a holomorphic family of Laurent polynomials over a neighborhood of $q = 0$, and
\begin{equation*}
W^{LF}_\mathcal{X}(q)=W^{LF}_Y(Q(q)).
\end{equation*}
\end{conjecture}

As $\mathcal{X}$ is Gorenstein, $\iota(\nu)$ is a positive integer for any $\nu\in\mathrm{Box}'$; in particular, we have $\iota(\nu)\geq1$. Recall that in the definition of the extended stacky fan (and hence $W^{LF}_\mathcal{X}$), we restricted to those $\nu$ with $\iota(\nu)\leq1$. Hence $\mathcal{X}$ is Gorenstein implies that $\iota(\nu_j)=1$ for $m<j\leq m'$. In particular, $W^{LF}_\mathcal{X}$ is summing over all $\beta\in\pi_2(\mathcal{X},L)$ with Chern-Weil Maslov index $\mu_{CW}(\beta)=2$. Moreover, if we write $\beta=\sum_{j=1}^m k_j\beta_j+\sum_{j=m+1}^{m'}k_j\beta_{\nu_j}+d$ with $k_j\in \Z_{\geq0}$ and $d\in H_2^\mathrm{eff}(\mathcal{X})$, then $\mu_{CW}(\beta)=2\sum_{j=1}^{m'} k_j+2c_1^{CW}(d)$. Since $\mathcal{X}$ is semi-Fano, $c_1^{CW}(d)\geq0$. Hence the condition $\mu_{CW}(\beta)=2$ implies that $\beta$ must be of one of the following forms:
\begin{enumerate}
\item[(1)] $\beta=\beta_j+d$ for $j=1,\ldots,m$ and $d\in H_2^\mathrm{eff}(\mathcal{X})$ with $c_1^{CW}(d)=0$, or
\item[(2)] $\beta=\beta_{\nu_j}+d$ for $j=m+1,\ldots,m'$ and $d\in H_2^\mathrm{eff}(\mathcal{X})$ with $c_1^{CW}(d)=0$.
\end{enumerate}
In view of this, the Lagrangian Floer superpotential of $\mathcal{X}$ can be expressed as
\begin{align*}
W^{LF}_\mathcal{X}= & \sum_{j=1}^m \left(\sum_{\stackrel{d\in H_2^\mathrm{eff}(\mathcal{X})}{c_1^{CW}(d)=0}}\sum_{l\geq0}\frac{1}{l!}
n_{1,l,\beta_j+d}^\mathcal{X}([\mathrm{pt}]_{L};\tau_\mathrm{tw},\ldots,\tau_\mathrm{tw})q^d\right) Z_j\\
& \quad + \sum_{j=m+1}^{m'} \left(\sum_{\stackrel{d\in H_2^\mathrm{eff}(\mathcal{X})}{c_1^{CW}(d)=0}}\sum_{l\geq0}\frac{1}{l!}
n_{1,l,\beta_{\nu_j}+d}^\mathcal{X}([\mathrm{pt}]_{L};\tau_\mathrm{tw},\ldots,\tau_\mathrm{tw})q^d \right) Z_j.
\end{align*}

As a result, the open CRC is equivalent to asserting the following equalities\footnote{We emphasize that these equalities are equalities between analytic functions (as oppose to formal power series).} between generating functions of open (orbifold) GW invariants for $\mathcal{X}$ and $Y$:
\begin{align}\label{openmatch1}
\sum_{\stackrel{d\in H_2^\mathrm{eff}(\mathcal{X})}{c_1^{CW}(d)=0}}\sum_{l\geq0}\frac{1}{l!}
n_{1,l,\beta_j+d}^\mathcal{X}([\mathrm{pt}]_{L};\tau_\mathrm{tw},\ldots,\tau_\mathrm{tw})q^d  = \sum_{\stackrel{\alpha\in H_2^\mathrm{eff}(Y)}{c_1(\alpha)=0}}n^Y_{\beta_j+\alpha}Q^\alpha,
\end{align}
for $j=1,\ldots,m$, and
\begin{align}\label{openmatch2}
\sum_{\stackrel{d\in H_2^\mathrm{eff}(\mathcal{X})}{c_1^{CW}(d)=0}}\sum_{l\geq0}\frac{1}{l!}
n_{1,l,\beta_{\nu_j}+d}^\mathcal{X}([\mathrm{pt}]_{L};\tau_\mathrm{tw},\ldots,\tau_\mathrm{tw})q^d = \sum_{\stackrel{\alpha\in H_2^\mathrm{eff}(Y)}{c_1(\alpha)=0}}n^Y_{\beta_j+\alpha}Q^\alpha.
\end{align}
for $j=m+1,\ldots,m'$, after analytic continuation of the generating functions for $Y$ and a change of variables $Q=Q(q)$.

\subsection{Relation to the closed CRC}\label{closedCRC_vs_openCRC}

The open CRC (Conjecture \ref{open_CRC}) is closely related to the closed crepant resolution conjecture. This can be best seen using the formulation due to Coates-Iritani-Tseng \cite{CIT09} (see also Coates-Ruan \cite{Coates-Ruan}). Let us first briefly recall their formulation.

In \cite{Givental04}, Givental proposed a symplectic formalism to understand Gromov-Witten theory. Let $\mathcal{Z}$ be either $\mathcal{X}$ or $Y$. Then let
$$\mathcal{H}_\mathcal{Z}:=H^*_\mathrm{orb}(\mathcal{Z};\Lambda)\otimes\C((z^{-1})),$$
where $\Lambda$ is a certain Novikov ring. This is an infinite dimensional symplectic vector space under the pairing
$$\Omega_\mathcal{Z}(f,g)=\mathrm{Res}_{z=0}(f(-z),g(z))_\mathcal{Z} dz,$$
where $(\cdot,\cdot)$ denotes the orbifold Poincar\'e pairing. Givental's Lagrangian cone for $\mathcal{Z}$ is a Lagrangian submanifold-germ $\mathcal{L}_\mathcal{Z}$ in the symplectic vector space $\mathcal{H}_\mathcal{Z}$ defined as the graph of the differential of the {\em genus 0 descendent GW potential} $\mathcal{F}^0_\mathcal{Z}$. It encodes all the genus zero (orbifold) GW invariants of $\mathcal{Z}$ and many relations in GW theory can be rephrased as geometric constraints on $\mathcal{L}_\mathcal{Z}$ \cite{Coates-Givental07,Givental04}.

The closed CRC in \cite{CIT09} was formulated as
\begin{conjecture}[Closed CRC; Conjecture 1.3 in \cite{CIT09}]\label{closed_CRC}
There exists a linear symplectic transformation $\mathbb{U}:\mathcal{H}_\mathcal{X} \to \mathcal{H}_Y$, satisfying certain conditions, such that after analytic continuations of $\mathcal{L}_\mathcal{X}$ and $\mathcal{L}_Y$, we have
$$\mathbb{U}(\mathcal{L}_\mathcal{X})=\mathcal{L}_Y.$$
\end{conjecture}

We have become aware recently that Conjecture \ref{closed_CRC} has now been proven in full generality in the toric case by Coates, Iritani and Jiang \cite{CIJ}, though we have not seen the details of their proof yet. On the other hand, the result \cite[Theorem  1.12]{G-W-wc} of Gonzalez and Woodward implies a relation between Gromov-Witten invariants of $\mathcal{X}$ and $Y$. This should be considered as proving a version of closed CRC. It is plausible that \cite[Theorem  1.12]{G-W-wc} can be used to deduce Conjecture 31, but at this point we do not know how to do this.

In practice, to prove this conjecture, one computes the symplectic transformation $\mathbb{U}$ by first analytically continuing the $I$-function $I_Y$ of $Y$ from a neighborhood of the large complex structure limit point for $Y$ (i.e. near $U=0$ in $\CM^Y_B$) to a neighborhood of the large complex structure limit point for $\mathcal{X}$ (i.e. near $y=0$ in $\CM^\mathcal{X}_B$), and then comparing it with the $I$-function $I_\mathcal{X}$ for $\mathcal{X}$. Since the coefficients of $I_Y$ are hypergeometric functions, one can use Mellin-Barnes integrals to perform this analytic continuation (as done by Borisov-Horja \cite{BH06}). Notice that the choice of branch cuts in the analytic continuation process always lead to an ambiguity in the construction of $\mathbb{U}$ (see \cite[Remark 3.10]{CIT09}). This is also what happens in the construction of our change of variables $Q=Q(q)$.

The relation between the open and closed CRC originates from the following construction of the change of variables $Q=Q(q)$ from the symplectic transformation $\mathbb{U}$: We first expand $\mathbb{U}^{-1}(I_Y)$ near the large complex structure limit point for $\mathcal{X}$. In terms of the coordinates $y\in \CM^\mathcal{X}_B$, we have
$$\mathbb{U}^{-1}(I_Y)=1+\frac{\Lambda(y)}{z}+O(z^{-2}).$$
The map $\Lambda(y)$ takes values in a neighborhood of the large radius limit point (i.e. $Q=0$) in $\CM^Y_A$. Then we define the change of variables $Q=Q(q)$ as the composition of the map $\Upsilon(y):=\exp\Lambda(y)$ induced by $\mathbb{U}$ and the inverse mirror map $y=y(q)$ for $\mathcal{X}$. Then we have
\begin{theorem}\label{closedCRC_implies_openCRC}
Assume that the open mirror theorems for $\mathcal{X}$ and $Y$ (Conjecture \ref{open_mirror_thm} and Theorem \ref{open_mirror_thm_Y} respectively) hold. Also assume that the closed CRC (Conjecture \ref{closed_CRC}) holds, with $\mathbb{U}(I_\mathcal{X})=I_Y$. Then the open crepant resolution conjecture (Conjecture \ref{open_CRC}) is true:
\begin{equation*}
W^{LF}_\mathcal{X}(q)=W^{LF}_Y(Q(q)),
\end{equation*}
via the change of variables $Q=Q(q)$ for the quantum parameters defined above.
\end{theorem}
\begin{proof}
The composition $U\circ\Upsilon$ of the mirror map $Q\mapsto U(Q)$ with the map $\Upsilon=\Upsilon(y)$ defined above gives a gluing of the B-model moduli spaces $\CM^Y_B$ with $\CM^\mathcal{X}_B$. This extends the family of LG superpotentials $W^{HV}_Y$ over a larger base which includes the neighborhood of the large complex structure limit point for $\mathcal{X}$ over which $W^{HV}_\mathcal{X}$ is defined. Moreover, by constructions,
$$W^{HV}_Y((U\circ\Upsilon)(y))=W^{HV}_\mathcal{X}(y)$$
since we have $\mathbb{U}(I_\mathcal{X})=I_Y$.

On the other hand, the open mirror theorems for $Y$ and $\mathcal{X}$ state that
\begin{align*}
W^{LF}_Y(Q) & = W^{HV}_Y(U(Q)),\textrm{ and}\\
W^{LF}_\mathcal{X}(q) & =W^{HV}_\mathcal{X}(y(q)).
\end{align*}
respectively. It follows that
$$W^{LF}_Y(Q(q))=W^{HV}_Y((U\circ\Upsilon)(y(q)))=W^{HV}_\mathcal{X}(y(q))=W^{LF}_\mathcal{X}(q),$$
which yields the open CRC.
\end{proof}
\begin{remark}$\mbox{}$
\begin{enumerate}
\item[(1)] One can also calculate the change of variables $Q=Q(q)$ by a direct analytic continuation of the mirror map for $Y$ using Mellin-Barnes integrals \cite{BH06}, which, by the open mirror theorem, corresponds to analytic continuation of the Lagrangian Floer superpotential $W^{LF}_Y$.
\item[(2)] As have been observed in \cite{CIT09}, the change of variables $y\mapsto (U\circ\Upsilon)(y)$ from $\CM^\mathcal{X}_B$ to $\CM^Y_B$ which appears in the proof does not necessarily preserve the flat structures near the large complex structure limits in the B-model moduli spaces. This was the case for the example $\mathcal{X}=\proj(1,1,1,3)$, $Y=\proj(K_{\proj^2}\oplus\mathcal{O}_{\proj^2})$ as has been demonstrated in \cite{CIT09}. Indeed this is also the case for $\proj(1,\ldots,1,n)$ whenever $n\geq3$.
\item[(3)] Suppose that the toric K\"ahler orbifold $(\mathcal{X},\omega)$ satisfies the Hard Lefschetz condition, i.e. $$\omega^k\cup_\mathrm{orb}:H^{n-k}_\mathrm{orb}(\mathcal{X})\to H^{n+k}_\mathrm{orb}(\mathcal{X}),$$
    where $\cup_\mathrm{orb}$ denotes the Chen-Ruan orbifold cup product, is an isomorphism for all $k\geq0$. An example is given by the weighted projective plane $\mathcal{X}=\proj(1,1,2)$. Then \cite[Theorem 5.10]{CIT09} implies that the symplectic transformation $\mathbb{U}$ can be written as
    $$\mathbb{U}=U_0+U_1z^{-1}+\cdots+U_Nz^{-N}$$
    for some $N\in\Z_{\geq0}$ and some linear maps $U_i:H^*_\mathrm{orb}(\mathcal{X};\C)\to H^*(Y;\C)$. In this case, the change of variables needed in the open CRC is simply given by $Q=U_0(q)$. See \cite[Section 9]{Coates-Ruan}.
\end{enumerate}
\end{remark}

\subsection{Specialization of quantum parameters and disc counting}\label{sec:specialization}

Ruan's original crepant resolution conjecture \cite{Ruan06} states that the small quantum cohomology ring of the crepant resolution $Y$ is isomorphic to the small quantum cohomology ring of $\mathcal{X}$ after analytic continuation of quantum parameters of $Y$ and specialization of the exceptional ones to certain roots of unity. Using the open CRC, we are able to give a new geometric interpretation of this specialization.

To begin with, we recall that, as a corollary of the results in Fukaya-Oh-Ohta-Ono \cite{FOOO10b}, we have a ring isomorphism
\begin{equation*}
QH^*(Y,Q)\cong Jac(W^{LF}_Y(Q)),
\end{equation*}
where the right hand side is the Jacobian ring
$$Jac(W^{LF}_Y):=\C[z_1^{\pm1},\ldots,z_n^{\pm1}]/\langle \partial_1W^{LF}_Y,\ldots,\partial_n W^{LF}_Y\rangle$$
of the Lagrangian Floer superpotential $W^{LF}_Y$. On the other hand, we expect that there is also a ring isomorphism between the small quantum {\em orbifold} cohomology of $\mathcal{X}$ and the Jacobian ring of the Lagrangian Floer superpotential of $\mathcal{X}$:
\begin{equation*}
QH_\mathrm{orb}^*(\mathcal{X},q)\cong Jac(W^{LF}_\mathcal{X}(q)).
\end{equation*}
These two results together with the open CRC (Conjecture \ref{open_CRC}) then implies that, after analytic continuation and the change of variables $Q=Q(q)$ for the quantum parameters, we have a ring isomorphism
\begin{equation*}
QH^*(Y,Q)\cong QH_\mathrm{orb}^*(\mathcal{X},q)
\end{equation*}
between the small quantum cohomology ring of $Y$ and the small {\em orbifold} quantum cohomology ring of $\mathcal{X}$.

Now the small quantum cohomology ring of the orbifold $\mathcal{X}$ is given by setting all the orbi-parameters $\tau_\mathrm{tw}$ to zero. Correspondingly, the change of variables $Q=Q(q)$ becomes
\begin{align}\label{specialization}
Q_a = \left\{\begin{array}{ll}
             e^{\langle c+f,d_a\rangle} q_a,&\ a=1,\ldots,r,\\
             e^{\langle c+f,d_a\rangle},&\ a=r+1,\ldots,r',
             \end{array}\right.
\end{align}
where $c\in H^2(Y;\C)$ is the class defined by
$$\mathbb{U}(\mathbf{1}_\mathcal{X})=\mathbf{1}_Y-cz^{-1}+O(z^{-2})$$
and $f\in H^2(Y;\C)$ is an exceptional class. Note that this is similar but not always the same as Ruan's CRC because $e^{\langle c+f,d_a\rangle}$ is not necessarily a root of unity. Hence this leads to a ``quantum corrected" version of Ruan's CRC. See \cite{Coates-Ruan} (in particular Section 8) for an excellent explanation of what is happening.

From the point of view of Lagrangian Floer theory and disc counting, setting $\tau_\mathrm{tw}=0$ corresponds to switching off all contributions from orbi-discs in the Lagrangian Floer superpotential $W^{LF}_\mathcal{X}$. In particular, all terms in the infinite sum
$$\sum_{\stackrel{d\in H_2^\mathrm{eff}(\mathcal{X})}{c_1^{CW}(d)=0}}\sum_{l\geq0}\frac{1}{l!}
n_{1,l,\beta_{\nu_j}+d}^\mathcal{X}([\mathrm{pt}]_{L};\tau_\mathrm{tw},\ldots,\tau_\mathrm{tw})q^d$$
will vanish because a holomorphic orbi-disc must have at least one interior orbifold marked point, so that the invariant $n_{1,l,\beta_{\nu_j}+d}^\mathcal{X}([\mathrm{pt}]_{L};\tau_\mathrm{tw},\ldots,\tau_\mathrm{tw})$ is nonzero only when $l>0$. By the open CRC, the corresponding terms in $W^{LF}_Y$, which correspond precisely to those discs meeting the exceptional divisors in $Y$, also vanish. Hence we conclude that:
\begin{theorem}\label{thm_specialization}
Suppose that the open CRC (Conjecture \ref{open_CRC}) holds. If we write the Lagrangian Floer superpotential of the crepant resolution $Y$ as $W^{LF}_Y=W^{LF,\mathrm{excep}}_Y+W^{LF,\mathrm{rest}}_Y$, where
$$W^{LF,\mathrm{excep}}_Y=\sum_{j=m+1}^{m'}\left(\sum_{\stackrel{\alpha\in H_2^\mathrm{eff}(Y)}{c_1(\alpha)=0}}
n^Y_{\beta_j+\alpha}Q^\alpha\right)Z_j$$
is the sum of terms coming from discs meeting the exceptional divisors in $Y$, then each term of $W^{LF,\mathrm{excep}}_Y$ vanishes after the change of variables and specialization \eqref{specialization}.
\end{theorem}

\section{A comparison theorem}\label{comparison_thm}

In this section, we derive an equality between open and closed invariants in the orbifold setting. We will first consider the case of Gorenstein toric Fano orbifolds (Theorem \ref{THM_open_closed}). The corresponding result (Theorem \ref{Thm_open_closed2}) for more general cases (i.e. semi-Fano and not necessarily Gorenstein) will be discussed at the end of this section.

\begin{theorem} \label{THM_open_closed}
Let $\mathcal{X}$ be a Gorenstein toric Fano orbifold (possibly non-compact) and $L$ a Lagrangian torus fiber. Suppose that there is a stable holomorphic orbifold disc in $\CM^{main}_{1,l}(L,\beta,\bx)$ for $\beta \in \pi_2(\mathcal{X},L)$, and further assume that $\mu_{CW}(\beta) = 2$. Here $\bx = (\mathcal{X}_{\nu_1},\ldots,\mathcal{X}_{\nu_l})$ for $\nu_i \in \{\nu\in Box' \mid \iota(\nu)=1\}$. Then there exist an explicit toric orbifold $\bar{\mathcal{X}}$ and an explicit homology class $\bar{\beta} \in H_2(\bar{\mathcal{X}};\Z)$ such that the following equality between open orbifold GW invariants of $\mathcal{X}$ and closed orbifold GW invariants of $\bar{\mathcal{X}}$ holds:
$$
n_{1,l,\beta}^\mathcal{X}([\mathrm{pt}]_{L};\mathbf{1}_{\nu_1},\ldots,\mathbf{1}_{\nu_l}) =
\langle [\mathrm{pt}]_{\bar{\mathcal{X}}}, \mathbf{1}_{\nu_1},\ldots,\mathbf{1}_{\nu_l} \rangle^{\bar{\mathcal{X}}}_{0,l+1,\bar{\beta}}
$$
where $[\mathrm{pt}]_{L} \in H^n(L;\rat)$ (resp. $[\mathrm{pt}]_{\bar{\mathcal{X}}} \in H^{2n}(\bar{X};\rat)$) denotes the point class of $L$ (resp. $\bar{\mathcal{X}}$).
\end{theorem}

The proof of Theorem \ref{THM_open_closed} is an adaptation of the proof of Theorem 1.1 in \cite{Chan10} and the proof of Proposition 4.4 in \cite{LLW10} to the orbifold setting. One major difference is that in the setting of \cite{Chan10, LLW10}, no interior insertions are allowed. In contrast, the open orbifold GW invariants considered in this paper are allowed to have interior orbi-insertions $\mathbf{1}_{\nu_i}$. Notice that the Divisor Axiom is not valid for orbi-insertions even for degree two.

In the orbifold case, bubbling components of the disc are constantly mapped to an orbifold point in $\mathcal{X}$ (by stability the bubbling components have to contain orbifold marked points such that the domain is stable), while in \cite{Chan10,LLW10}, bubbling components are mapped to rational curves with Chern number zero.  The assumption that rational curves do not deform away was required.  On the other hand, orbi-strata have the advantage that it is invariant under torus action and hence automatically cannot deform away.

Also notice that the toric modification $\bar{\mathcal{X}}$ can be stacky even when $\mathcal{X}$ is non-stacky, because the newly added vector $\bb_\infty$ may not be primitive. This situation does not occur in the manifold case since a basic disc class always corresponds to a primitive vector in the manifold setting.

To construct $(\bar{\mathcal{X}},\bar{\beta})$ explicitly, we need the following proposition:

\begin{prop} \label{Prp_disk_form}
Assume the notations and conditions in Theorem \ref{THM_open_closed}. Then every stable holomorphic orbi-disc $u \in \CM^{main}_{1,l}(L,\beta,\bx)$ representing the class $\beta \in \pi_2(\mathcal{X},L)$ satisfies the following:
\begin{enumerate}
\item The domain of $u$ is connected and it consists of one disc component $\bD$ and possibly a connected rational curve $\mathcal{C}$ consisting of (orbi-)sphere components.
\item $u_0 := u|_\bD$ represents $\beta$ and $\beta$ must be a basic disc class.
\item When $l=0,1$, $\Dom (u) = \bD$.  When $l\geq 2$,  $\mathcal{C}$ is non-empty.
\item Suppose that $l \geq 2$. When $\beta$ is a basic disc class represented by a smooth holomorphic disc, $\bD$ and $\mathcal{C}$ intersect at an ordinary nodal point.  When $\beta$ is a basic disc class represented by a holomorphic orbi-disc, $\bD$ and $\mathcal{C}$ intersect at an orbifold nodal point.  Denote the image of the (orbifold) nodal point under $u$ by $p \in \mathcal{X}$.
\item $u|_\mathcal{C} = p \in \mathcal{X}$.
\end{enumerate}
\end{prop}

\begin{proof}
By the definition of a stable holomorphic orbi-disc, the domain of $u$ must consist of (orbi-)disc components $\bD_1$, \ldots, $\bD_j$ and (orbi-)sphere components $\mathcal{C}_1, \ldots, \mathcal{C}_k$. One has
$$\mu_{CW}(u) = \mu_{CW} (u|_{\bD_1}) + \ldots + \mu_{CW}(u|_{\bD_j}) + 2 c_1^{CW}(u|_{\mathcal{C}_1}) + \ldots + 2 c_1^{CW}(u|_{\mathcal{C}_k}).$$
Since every class represented by holomorphic (orbi-)discs are generated by basic disc classes and each basic class has $\mu_{CW} \geq 2$ because $\mathcal{X}$ is Gorenstein, we have $\mu_{CW}([u|_{\bD_i}]) \geq 2$ for all $i$. On the other hand, by the assumption that $\mathcal{X}$ is Fano, we have $c_1^{CW}(u|_{\mathcal{C}_i}) \geq 0$ and equality holds if and only if $u$ is constant on $\mathcal{C}_i$ for each $i$. Therefore the condition $\mu_{CW}(u) = 2$ forces us to have $j=1$ and $u$ is constant on each $\mathcal{C}_i$. This proves (1).

Denote $\bD = \bD_1$ so that $\beta = [u|_\bD]$. Now the class of $u_0:=u|_\bD$ is generated by basic disc classes. But since $\mu_{CW}([u_0])=2$, $[u_0]$ has to be one of the basic disc classes. This proves (2).

Suppose that $l=0$ or $1$. Then by the stability of the map $u$, the domain $\Dom (u)$ cannot have constant (orbi-)sphere components, and hence $\Dom (u) = \bD$. When $l \geq 2$, since $\bD$ has at most one interior orbi-marked point (by the classification, i.e. Theorem \ref{classification_orbidiscs}, for holomorphic orbi-discs), $\Dom (u)$ must consist of some (orbi-)sphere components. This proves (3).

Assuming that $l \geq 2$. Then there are two cases: $u|_\bD$ is a smooth holomorphic disc, or $u|_\bD$ is a holomorphic orbi-disc. Since $[u|_\bD]=\beta$ is basic, in both cases $u|_\bD$ intersect $\mathcal{X}\setminus\mathcal{X}^\circ$ at exactly one interior point $z$ (here $\mathcal{X}^\circ := \mathcal{X}\setminus\bigcup_{j=1}^m D_j \cong (\C^*)^n$ is the open dense toric stratum of $\mathcal{X}$). Let $p=u(z)\in \mathcal{X}$ be the point of intersection. Notice that all orbifold points of $\mathcal{X}$ lie in $\mathcal{X}\setminus\mathcal{X}^\circ$. Since an orbifold marked point must be mapped to an orbifold point of $\mathcal{X}$ and $u$ is constant on $\mathcal{C}_i$, we have $u|_{\mathcal{C}_i} \equiv p$. This forces $\mathcal{C} := \bigcup_i \mathcal{C}_i$ to be a connected rational curve and $\bD$ intersects $\mathcal{C}$ at $z \in \bD$. When $u|_\bD$ is a smooth holomorphic disc, the intersection has to be an ordinary nodal point. When $u|_\bD$ is an holomorphic orbi-disc, $z$ is an orbifold point, and so the intersection has to be an orbifold nodal point. This proves (4) and (5).
\end{proof}

We are now ready to construct $(\bar{\mathcal{X}}, \bar{\beta})$, assuming the setting of Theorem \ref{THM_open_closed}.

\begin{defn} \label{Defn_barX}
Assume the notations and conditions in Theorem \ref{THM_open_closed}. By Proposition \ref{Prp_disk_form}, $\beta$ must be one of the basic disc classes.  Let $\bb_0 = \partial \beta \in N$ and consider $\bb_\infty := -\bb_0$.

Let $C = \langle \bb_{i_1}, \ldots, \bb_{i_l} \rangle_{\real_{\geq 0}}$ be the minimal cone of $\Sigma$ which contains $\bb_\infty$.  If $l = 1$ (which means $\bb_\infty$ is contained in a ray of $\Sigma$), one replaces $\bb_{i_1}$ by $\bb_\infty$ and obtains a new stacky fan $\bar{\Sigma}$.  If $l > 1$, consider the subdivision of $C$ by subcones $\langle \bb_\infty, \bb_{i_1}, \ldots, \widehat{\bb_{i_j}}, \ldots, \bb_{i_l} \rangle_{\real_{\geq 0}}$ for $j = 1, \ldots, l$.  This subdivision induces a subdivision of any cone $\tilde{C} = \langle \bb_{i_1}, \ldots, \bb_{i_l}, \bb_{k_1}, \ldots, \bb_{k_p} \rangle_{\real_{\geq 0}}$ containing $C$, where the subcones are given by $\langle \bb_\infty, \bb_{i_1}, \ldots, \widehat{\bb_{i_j}}, \ldots, \bb_{i_l}, \bb_{k_1}, \ldots, \bb_{k_p} \rangle_{\real_{\geq 0}}$.  Thus one obtains a new stacky fan $\bar{\Sigma}$ which is a refinement of $\Sigma$, and whose set of stacky vectors is a union of that of $\Sigma$ and $\{\bb_\infty\}$.  Then let $\bar{\mathcal{X}}$ be the toric orbifold associated to the stacky fan $\bar{\Sigma}$.

Denote by $\beta_\infty \in \pi_2(\bar{\mathcal{X}}, L)$ the basic disc class corresponding to $\bb_\infty$.  Since $\partial (\beta + \beta_\infty) = \bb_0 + (-\bb_0) = 0$, $\bar{\beta}:= \beta + \beta_\infty$ belongs to $H_2(\bar{\mathcal{X}};\Z)$.  This finishes the construction of $(\bar{\mathcal{X}}, \bar{\beta})$.
\end{defn}

We shall now proceed to the proof of Theorem \ref{THM_open_closed}. Orbifold smoothness is used here instead of ordinary smoothness for manifolds.

\begin{proof}[Proof of Theorem \ref{THM_open_closed}]
The strategy is to prove that the open moduli and the closed moduli have the same Kuranishi structures.

First of all, let us set up some notations. Let $\bx = (\mathcal{X}_{\nu_1},\ldots,\mathcal{X}_{\nu_l})$ be the type of $\beta$. Recall that this records the twisted sectors that the interior orbifold marked points of a map representing $\beta$ pass through.  Denote the twisted sector of $\bar{\mathcal{X}}$ which corresponds to $\mathcal{X}_{\nu_i}$ by $\bar{\mathcal{X}}_{\nu_i}$. Then set $\bar{\bx} = (\bar{\mathcal{X}},\bar{\mathcal{X}}_{\nu_1},\ldots,\bar{\mathcal{X}}_{\nu_l})$ (recall that $\bar{\mathcal{X}}=\bar{\mathcal{X}}_0$ is the trivial twisted sector).

Denote by $\CM^{\mathrm{op}}_{1,l}(\beta,\bx):=\CM^{main}_{1,l}(L,\beta,\bx)$ the moduli space of stable maps from genus 0 bordered orbifold Riemann surfaces with one boundary marked point and $l$ interior orbifold marked points of type $\bx$ representing the class $\beta$, and by $\CM^{\mathrm{cl}}_{l+1}(\bar{\beta},\bar{\bx})$ the moduli space of stable maps from genus zero nodal orbifold curves with one smooth marked point and $l$ orbifold marked points of type $\bar{\mathbf{x}}$ representing the class $\bar{\beta}$.

Fix a point $p \in L$ and define
$$\CM^{\mathrm{op}}_{1,l}(\beta,\bx;p) := \CM^{\mathrm{op}}_{1,l}(\beta,\bx)_{ev_0}\times_\iota \{p\} $$
where we use the evaluation map $ev_0:\CM^{\mathrm{op}}_{1,l}(\beta,\bx) \to L$ at the boundary marked point and the inclusion map $\iota: \{p\} \hookrightarrow L$ in the fiber product. Similarly, we define
$$\CM^{\mathrm{cl}}_{l+1}(\bar{\beta},\bar{\bx};p) := \CM^{\mathrm{cl}}_{l+1} (\bar{\beta}, \bar{\bx})_{ev_0}\times_{\bar{\iota}} \{p\}$$
where we use the evaluation map $ev_0:\CM^{\mathrm{cl}}_{l+1} (\bar{\beta}, \bar{\bx}) \to \bar{\mathcal{X}}$ at the smooth marked point and the inclusion map $\bar{\iota}: \{p\} \hookrightarrow \bar{\mathcal{X}}$ in the fiber product.

$\CM^{\mathrm{op}}_{1,l}(\beta,\bx)$ is equipped with an oriented Kuranishi structure with tangent bundle. By Lemma A1.39 of \cite{FOOO_I,FOOO_II} on fiber products, this induces an oriented Kuranishi structure with tangent bundle on $\CM^{\mathrm{op}}_{1,l}(\beta,\bx;p)$.  Similarly $\CM^{\mathrm{cl}}_{l+1}(\bar{\beta}, \bar{\bx})$ and $\CM^{\mathrm{cl}}_{l+1}(\bar{\beta},\bar{\bx};p)$ are both equipped with an oriented Kuranishi structure with tangent bundle.

Let us begin by computing the virtual dimensions. The (real) virtual dimension of $\CM^{\mathrm{op}}_{1,l}(\beta,\bx;p)$ is given by
$$ \mu_{CW}(\beta) + 1 + 2l - 3 - 2\iota(\bx) = \mu_{CW}(\beta) + 2l - 2\iota(\bx) - 2,$$
where $\iota(\bx)=\sum_{i=1}^l \iota(\nu_i)$; while the (real) virtual dimension of $\CM^{\mathrm{cl}}_{l+1}(\bar{\beta},\bar{\bx};p)$ is given by
$$ 2 c_1^{CW}(\bar{\beta}) + 2(l+1) - 6 - 2\iota(\bx) = 2 c_1^{CW}(\bar{\beta}) + 2l - 2\iota(\bx) - 4.$$
Now
$$2 c_1^{CW}(\bar{\beta}) = \mu_{CW}(\beta) + \mu_{CW}(\beta_\infty) = \mu_{CW}(\beta) + 2,$$
where we have $\mu_{CW}(\beta_\infty) = 2$ because $\beta_\infty$ is a smooth basic disc class. Thus we see that they have the same virtual dimension (in fact, since $\mu_{CW}(\beta)=2$ and $\iota(\bx)=l$, they both have virtual dimension zero). In the following we prove that they are isomorphic as Kuranishi spaces. The proof is divided into 3 steps:\\

\noindent\emph{Step 1}: We have
$$\CM^{\mathrm{op}}_{1,l}(\beta,\bx;p) = \CM^{\mathrm{cl}}_{l+1}(\bar{\beta},\bar{\bx};p)$$
as a set.

\noindent\emph{Proof}:
By Proposition \ref{Prp_disk_form}, the domain of every stable map $u$ with one boundary marked point and $l$ interior orbifold marked points representing $\beta$ consists of an orbi-disc component representing $\beta$ and some constant orbi-sphere components. By Cho-Poddar \cite{CP}, $\beta$ corresponds to a twisted sector $\mathcal{X}_{\nu}$ of $\mathcal{X}$ and the evaluation map $ev_0:\CM^{main}_{1,1}(L,\beta,\mathcal{X}_\nu) \to L$ at the boundary marked point is a diffeomorphism. This means that there exists a unique (up to automorphisms of domain) holomorphic orbi-disc $u_0$ representing $\beta$ and passing through $p\in L$. Thus every such stable disc $u \in \CM^{\mathrm{op}}_{1,l}(\beta,\bx;p)$ has the same holomorphic orbi-disc component $u_0$. In conclusion, $u$ is $u_0$ attached with a constant rational orbi-curve at its only interior orbi-point.

Let $\bar{\beta}:= \beta + \beta_\infty \in H_2(\bar{\mathcal{X}};\Z)$. By the maximum principle, any rational orbi-curve representing $\bar{\beta}$ passing through $p\in\bar{\mathcal{X}}$ is unique (again up to automorphisms of domain); we call this curve $\bar{u}_0$.  Now let $\bar{u} \in \CM^{\mathrm{cl}}_{l+1}(\bar{\beta},\bar{\bx};p)$. Applying the maximum principle to $\bar{u}$ shows that any component of $u$ passing through $p$ must be $\bar{u}_0$.  Since $\bar{u}$ passes through $p$, it must contain such a component.  Moreover, since $\bar{u}$ and $\bar{u}_0$ have the same $c_1^{CW}$, and every non-trivial rational curve has $c_1^{CW} > 0$ because $\mathcal{X}$ is Fano, the restrictions of $\bar{u}$ to all the other orbi-sphere components are constant maps. By connectedness they have to be mapped to the same point. Moreover, by the stability of $\bar{u}$, they are mapped to the image of the unique orbi-point of $\bar{u}_0$. In conclusion, $\bar{u}$ is $\bar{u}_0$ attached with a constant rational orbi-curve at its only orbi-point.

Now it is not difficult to see that there is a one-to-one correspondence between  $\CM^{\mathrm{op}}_{1,l}(\beta,\bx;p)$ and $\CM^{\mathrm{cl}}_{l+1}(\bar{\beta},\bar{\bx};p)$: Any stable map $u \in \CM^{\mathrm{op}}_{1,l}(\beta,\bx;p)$ is given by $u_0$ attached with a constant rational orbi-curve at the unique interior orbi-point. We associate to it the stable map given by $\bar{u}_0$ attached with the same constant rational orbi-curve at its only orbi-point, which is an element of $\CM^{\mathrm{cl}}_{l+1}(\bar{\beta},\bar{\bx};p)$. Conversely any $\bar{u}
\in \CM^{\mathrm{cl}}_{l+1}(\bar{\beta},\bar{\bx};p)$ is $\bar{u}_0$ attached with a constant rational orbi-curve at its only orbi-point, and it can be associated to an element of $\CM^{\mathrm{op}}_{1,l}(\beta,\bx;p)$ in the same way.\\

\noindent\emph{Step 2}: We have the following equality between virtual cycles
$$[\CM^{\mathrm{op}}_{1,l}(\beta,\bx;p)]^{vir} = \iota^* [\CM^{\mathrm{op}}_{1,l}(\beta,\bx)]^{vir},$$
where $\iota: \{p\} \hookrightarrow L$ is the inclusion map. Similarly, we have
$$[\CM^{\mathrm{cl}}_{l+1}(\bar{\beta},\bar{\bx};p)]^{vir} =  \bar{\iota}^* [\CM^{\mathrm{cl}}_{l+1}(\bar{\beta},\bar{\bx})]^{vir},$$
where $\bar{\iota}: \{p\} \hookrightarrow \bar{\mathcal{X}}$ is the inclusion map.

\noindent\emph{Proof}: This follows directly from Lemma A1.43 of \cite{FOOO_I,FOOO_II}.\\

\noindent\emph{Step 3}: The Kuranishi structure on $\CM^{\mathrm{op}}_{1,l}(\beta,\bx;p)$ is the same as that on $\CM^{\mathrm{cl}}_{l+1}(\bar{\beta},\bar{\bx};p)$, and so
$$ev_{0*}[\CM^{\mathrm{op}}_{1,l}(\beta,\bx;p)]^{vir} = ev_{0*}[\CM^{\mathrm{cl}}_{l+1}(\bar{\beta},\bar{\bx};p)]^{vir}$$
as cycles in $H_0(\{p\};\rat)\cong\rat$. It then follows that the open orbifold GW invariant
$$n_{1,l,\beta}^\mathcal{X}([\mathrm{pt}]_{L};\mathbf{1}_{\nu_1},\ldots,\mathbf{1}_{\nu_l}) = ev_{0*}[\CM^{\mathrm{op}}_{1,l}(\beta,\bx;p)]^{vir}$$
is equal to the closed orbifold GW invariant
$$\langle [\mathrm{pt}]_{\bar{\mathcal{X}}}, \mathbf{1}_{\nu_1},\ldots,\mathbf{1}_{\nu_l} \rangle^{\bar{\mathcal{X}}}_{0,l+1,\bar{\beta}} = ev_{0*}[\CM^{\mathrm{cl}}_{l+1}(\bar{\beta},\bar{\bx};p)]^{vir}.$$

\noindent\emph{Proof}: Let $[\bar{u}] \in \CM^{\mathrm{cl}}_{l+1}(\bar{\beta},\bar{\bx};p)$ which corresponds to the element $[u] \in \CM^{\mathrm{op}}_{1,l}(\beta,\bx;p)$ by \emph{Step 1}.  $u$ consists of one disc component $u_0$ and a rational curve component.  The key observation is that since $u_0$ is regular, the obstruction merely comes from the rational curve component.  Similarly $\bar{u}$ consists of one smooth component $\bar{u}_0$ and the same rational curve component, and the obstruction again merely comes from this rational curve component.  Thus obstructions of $u$ can be identified with obstructions of $\bar{u}$ so that they have the same Kuranishi structures. Thus their virtual fundamental cycles are identical.

To make the above argument precise, let us briefly review the construction of Kuranishi structures in this situation. One has a Kuranishi chart
$$(V_{\mathrm{op}}, E_{\mathrm{op}}, \Gamma_{\mathrm{op}}, \psi_{\mathrm{op}}, s_{\mathrm{op}})$$
around $u$ which is constructed as follows \cite{FOOO_I, FOOO_II}.  Let
$$D_{u} \bar{\partial}:W^{1,p}(\mathrm{Dom}(u), u^*(T\mathcal{X}), L) \to W^{0,p}(\mathrm{Dom}(u), u^*(T\mathcal{X}) \otimes \Lambda^{0,1})$$
be the linearized Cauchy-Riemann operator at $u$.
\begin{enumerate}
\item
$\Gamma_{\mathrm{op}}$ is the finite automorphism group of $u$.

\item
$E_{\mathrm{op}}$ is the obstruction space which is the finite dimensional cokernel of the linearized Cauchy-Riemann operator $D_{u} \bar{\partial}$.  For the purpose of the next step of the construction, it is identified (in a non-canonical way) with a subspace of $W^{0,p}(\mathrm{Dom}(u), u^*(T\mathcal{X}) \otimes \Lambda^{0,1})$ as follows.  Denote by $\bD$ and $\mathcal{C}_1, \ldots, \mathcal{C}_l$ the (orbi-)disc and (orbi-)sphere components of $\mathrm{Dom}(u)$ respectively.   Take non-empty open subsets $W_0 \subset \bD$ and $W_i \subset \mathcal{C}_i$ for $i = 1, \ldots, l$.  Then by the unique continuation theorem, there exist finite dimensional subspaces $E_i \subset C_0^\infty (W_i, u^*(T\mathcal{X})\otimes \Lambda^{0,1})$ such that
$$\mathrm{Im} (D_{u} \bar{\partial}) \oplus E_{\mathrm{op}} = W^{0,p}(\mathrm{Dom}(u), u^*(T\mathcal{X}) \otimes \Lambda^{0,1})$$
and $E_{\mathrm{op}}$ is invariant under $\Gamma_{\mathrm{op}}$, where
$$E_{\mathrm{op}} := E_0 \oplus \ldots \oplus E_l.$$

\item
$\tilde{V}_{\mathrm{op}}$ is taken to be (a neighborhood of $0$ of) the space of first order deformations $\phi$ of $u$ which satisfies the linearized Cauchy-Riemann equation modulo elements in $E$, that is,
$$ D_{u} \bar{\partial} \phi \equiv 0 \,\, \mod E. $$
Such deformations may come from deformations of the map or deformations of complex structures of the domain.  More precisely,
$$\tilde{V}_{\mathrm{op}} = V_{\mathrm{op}}^{\mathrm{map}} \times V_{\mathrm{op}}^{\mathrm{dom}}$$
where $V_{\mathrm{op}}^{\mathrm{map}}$ is a neighborhood of zero in the kernel of the linear map
$$[D_{u} \bar{\partial}]:W^{1,p}(\mathrm{Dom}(u), u^*(T\mathcal{X}), L) \to W^{0,p}(\mathrm{Dom}(u), u^*(T\mathcal{X}) \otimes \Lambda^{0,1})/E_{\mathrm{op}}.$$
We remark that since $\mathrm{Dom}(u)$ is automatically stable in our case, there is no infinitesimal automorphism of the domain that in general one needs to quotient out.

$V_{\mathrm{op}}^{\mathrm{dom}}$ is a neighborhood of zero in the space of deformations of $\mathcal{C}=\bigcup_i \mathcal{C}_i$ which is the rational curve component of $\mathrm{Dom}(u)$ consisting of $l$ orbifold marked points. They consist of two types: one is deformations of each stable component (in this genus zero case, it means movements of special points in each component), and the other one is smoothing of nodes between components.  That is,
$$V_{\mathrm{op}}^{\mathrm{dom}} = V_{\mathrm{op}}^{\mathrm{cpnt}} \times V_{\mathrm{op}}^{\mathrm{smth}}$$
where $V_{\mathrm{op}}^{\mathrm{cpnt}}$ is a neighborhood of zero in the space of deformations of components of $C$, and $V_{\mathrm{op}}^{\mathrm{smth}}$ is a neighborhood of zero in the space of smoothing of the (orbifold) nodes (each node contribute to a one-dimensional family of smoothing).
Each element $\mathscr{D} \in V_{\mathrm{op}}^{\mathrm{dom}}$ corresponds to a stable holomorphic orbi-disc which is of the form $\bD \cup \tilde{\mathcal{C}}$, where $\bD$ is an orbi-disc with one boundary marked point and one interior orbifold marked point, and $\tilde{\mathcal{C}}$ is a rational curve with $l$ interior orbifold marked point, such that $\bD$ and $\tilde{\mathcal{C}}$ intersect at a nodal orbifold point.  By abuse of notation the orbi-disc is also denoted by $\mathscr{D}$, which serves as the domain of the deformed map in this context.

\item
$\tilde{s}_{\mathrm{op}}:\tilde{V}_{\mathrm{op}} \to E_{\mathrm{op}}$ is a transversal $\Gamma_{\mathrm{op}}$-equivariant perturbed zero-section of the trivial bundle $E_{\mathrm{op}} \times \tilde{V}_{\mathrm{op}}$ over $\tilde{V}_{\mathrm{op}}$.  By \cite{FOOO1} or \cite{CP}, this can be chosen to be $T^n$-equivariant.

\item There exists a continuous family of smooth maps $\rho^{\mathrm{op}}_{\phi}: (\mathscr{D},\partial\mathscr{D}) \to (\mathcal{X},L)$ for $\phi \in \tilde{V}_{\mathrm{op}}$ such that it solves the inhomogeneous Cauchy-Riemann equation: $\bar{\partial} \rho^{\mathrm{op}}_{\phi} = \tilde{s}_{\mathrm{op}}(\phi).$  Set
$$V_{\mathrm{op}} := \{\phi \in \tilde{V}_{\mathrm{op}}: ev_0 (\rho^{\mathrm{op}}_{\phi}) = p\}$$
where $ev_0$ is the evaluation map at the domain boundary marked point.  Then set $s_{\mathrm{op}} := \tilde{s}_{\mathrm{op}}|_{V_{\mathrm{op}}}$.

\item
$\psi_{\mathrm{op}}$ is a map from $s_{\mathrm{op}}^{-1} (0) / \Gamma_{\mathrm{op}}$ onto a neighborhood of $[u] \in \CM^{\mathrm{op}}_{1,l}(\beta,\bx;p)$.
\end{enumerate}

In Item (2) of the above construction, since the disc component $u_0$ of $u$ is unobstructed, we may take $E_0 = 0$ so that $E_{\mathrm{op}}$ can be taken to be of the form $E_{\mathrm{op}} = 0 \oplus E_1 \oplus \ldots \oplus E_l$. After this choice, we argue that $(V_{\mathrm{op}}, E_{\mathrm{op}}, \Gamma_{\mathrm{op}}, \psi_{\mathrm{op}}, s_{\mathrm{op}})$ can be identified with a Kuranishi chart $(V_{\mathrm{cl}}, E_{\mathrm{cl}}, \Gamma_{\mathrm{cl}}, \psi_{\mathrm{cl}}, s_{\mathrm{cl}})$ around the corresponding closed curve $\bar{u}$.

\begin{enumerate}
\item From the construction of the one-to-one correspondence between $u$ and $\bar{u}$, we see that $u$ and $\bar{u}$ have the same automorphism group, i.e.
    $$\Gamma_{\mathrm{cl}} = \Gamma_{\mathrm{op}}.$$

\item Notice that $\bar{u}_0$ has trivial obstruction.  Also all the other components are the same for $u$ and $\bar{u}$ so that $D_u \bar{\partial} = D_{\bar{u}} \bar{\partial}$ on these components.  It follows that
$$ \mathrm{Im} (D_{\bar{u}} \bar{\partial}) \oplus (0 \oplus E_1 \oplus \ldots \oplus E_l) = W^{0,p}(\mathrm{Dom}(\bar{u}), \bar{u}^*(T\mathcal{X}) \otimes \Lambda^{0,1}). $$
Thus we may also take $E_{\mathrm{cl}} = 0 \oplus E_1 \oplus \ldots \oplus E_l$.

\item We set $\tilde{V}_{\mathrm{cl}} = V_{\mathrm{cl}}^{\mathrm{map}} \times V_{\mathrm{cl}}^{\mathrm{dom}}$, where $V_{\mathrm{cl}}^{\mathrm{map}}$ and $V_{\mathrm{cl}}^{\mathrm{dom}}$ are defined in a similar way as in the open case.  The only difference is that $\Dom(\bar{u})$ has exactly one unstable component, namely, $\Dom(\bar{u}_0)$ and we need to define $V^{\mathrm{map}}_{\mathrm{cl}}$ to be a quotient of the kernel of the linear map $[D_{\bar{u}} \bar{\partial}]$ by the space of infinitestimal automorphisms of this unstable component.

    Since we have chosen $E_{\mathrm{cl}} = 0 \oplus E_1 \oplus \ldots \oplus E_l$, $V_{\mathrm{cl}}^{\mathrm{map}}$ consists of first order deformations which is holomorphic when restricted to the component $\Dom(\bar{u}_0)$.  Restrictions of such deformations to $\Dom(u)$ to give elements in $V_{\mathrm{op}}^{\mathrm{map}}$.  Conversely, since $E_{\mathrm{op}} = 0 \oplus E_1 \oplus \ldots \oplus E_l$, the first order deformations in $V_{\mathrm{op}}^{\mathrm{map}}$ are holomorphic when restricted to the disc component.  Thus they can be extended to $\Dom(\bar{u})$ to give elements in $V_{\mathrm{cl}}^{\mathrm{map}}$.  This establishes an identification between $V_{\mathrm{cl}}^{\mathrm{map}}$ and $V_{\mathrm{op}}^{\mathrm{map}}$.  Also $V_{\mathrm{cl}}^{\mathrm{dom}} = V_{\mathrm{op}}^{\mathrm{dom}}$  which is the first-order deformation of the rational curve component $\mathcal{C}$.  Hence we have
    $$\tilde{V}_{\mathrm{op}} = \tilde{V}_{\mathrm{cl}}.$$

\item With the above identification, $\tilde{s}_{\mathrm{op}}$ is identified with a section $\tilde{s}_{\mathrm{cl}}: \tilde{V}_{\mathrm{cl}} \to E_{\mathrm{cl}}.$

\item Again, since we have chosen $E_{\mathrm{op}} = 0 \oplus E_1 \oplus \ldots \oplus E_l$, it follows from $\bar{\partial} \rho^{\mathrm{op}}_{\phi} = \tilde{s}_{\mathrm{op}}(\phi)$ that $\rho^{\mathrm{op}}_{\phi}|_{\Delta}$ is holomorphic.  Together with $ev_0 (\rho^{\mathrm{op}}_{\phi}) = p$, we see that $\rho^{\mathrm{op}}_{\phi}|_{\bD} = u_0$.  Thus $\rho^{\mathrm{op}}_{\phi}$ extends to give a family $\rho^{\mathrm{cl}}_{\phi}$ for $\phi \in \tilde{V}_{\mathrm{cl}}$ which satisfies $\bar{\partial} \rho^{\mathrm{cl}}_{\phi} = \tilde{s}_{\mathrm{cl}}(\phi)$.  Set
    $$V_{\mathrm{cl}} := \{\phi \in \tilde{V}_{\mathrm{cl}}: ev_0 (\rho^{\mathrm{cl}}_{\phi}) = p\},$$
    where $ev_0$ is the evaluation map at the domain smooth marked point.  Then define
    $$s_{\mathrm{cl}} := \tilde{s}_{\mathrm{cl}}|_{V_{\mathrm{cl}}}.$$

\item From the above construction, $s_{\mathrm{op}}^{-1} (0) / \Gamma_{\mathrm{op}}$ can be identified with $s_{\mathrm{cl}}^{-1} (0) / \Gamma_{\mathrm{cl}}$.  Then $\psi_{\mathrm{op}}$ can be identified as a map $\psi_{\mathrm{cl}}$ which maps $s_{\mathrm{cl}}^{-1} (0) / \Gamma_{\mathrm{cl}}$ onto a neighborhood of $[\bar{u}] \in \CM^{\mathrm{cl}}_{l+1}(\bar{\beta},\bar{\bx};p)$.
\end{enumerate}

In conclusion, a Kuranishi neighborhood of $[u]$ can be identified with a Kuranishi neighborhood of $[\bar{u}]$. As a result, the Kuranishi structure on $\CM^{\mathrm{op}}_{1,l}(\beta,\bx;p)$ and that on $\CM^{\mathrm{cl}}_{l+1} (\bar{\beta},\bar{\bx};p)$ are identical. This completes the proof of Theorem \ref{THM_open_closed}.
\end{proof}

For simplicity we have made stronger assumptions in Theorem \ref{THM_open_closed} than required. Indeed the above argument applies to more general situations described as follows:

\begin{theorem} \label{Thm_open_closed2}
Let $\mathcal{X}$ be a semi-Fano toric orbifold (possibly non-compact) and let $L$ be a Lagrangian torus fiber of $\mathcal{X}$.  Let $\beta = \beta_0 + d \in \pi_2(\mathcal{X},L)$ be represented by a stable holomorphic (orbi-)disc with one boundary marked point and $l$ interior orbifold marked points and passing through $l$ non-trivial twisted sectors $\mathcal{X}_{\nu_i}$ for $i = 1, \ldots, l$, where $\beta_0$ is a basic disc class and $d$ is represented by a rational orbi-curve with $c_1^{CW}(d) = 0$.  When $\beta_0$ is a basic smooth disc class, let $S$ be the toric divisor that it passes through.  When $\beta_0$ is a basic orbi-disc class, let $S$ be the support of the twisted sector that it passes through.

Assume that each component of an (orbifold) rational curve $\mathcal{C} \subset \bar{\mathcal{X}}$ representing $d$ has $c_1^{CW} = 0$; if $\mathcal{C}$ intersects $S$, then $\mathcal{C}$ is contained in $\bar{\mathcal{X}}\setminus D_{\infty}$.  Here, $\bar{\mathcal{X}}$ is the toric orbifold birational to $\mathcal{X}$ constructed in Definition \ref{Defn_barX}, and $D_{\infty}$ is the divisor corresponding to $\bb_\infty$ involved in the construction. $S \subset \mathcal{X}$ is identified as a subset of $\bar{\mathcal{X}}$ by the birational map between $\mathcal{X}$ and $\bar{\mathcal{X}}$. Then we have the equality between open and closed orbifold GW invariants
$$
n_{1,l,\beta}^\mathcal{X}([\mathrm{pt}]_{L};\mathbf{1}_{\nu_1},\ldots,\mathbf{1}_{\nu_l}) =
\langle [\mathrm{pt}]_{\bar{\mathcal{X}}}, \mathbf{1}_{\nu_1},\ldots,\mathbf{1}_{\nu_l} \rangle^{\bar{\mathcal{X}}}_{0,l+1,\bar{\beta}}
$$
where $\bar{\beta}=\beta_0+\beta_\infty+d\in H_2(\bar{\mathcal{X}};\Z)$, $[\mathrm{pt}]_{L} \in H^n (L;\rat)$ denotes the point class of $L$ and $[\mathrm{pt}]_{\bar{\mathcal{X}}} \in H^{2n}(\bar{\mathcal{X}};\rat)$ denotes the point class of $\bar{\mathcal{X}}$.
\end{theorem}

Note that the cohomological degrees of the interior orbi-insertions $\mathbf{1}_{\nu_i}$ are {\em not} restricted to two here.

To prove Theorem \ref{Thm_open_closed2}, we first need to show that basic disc classes are {\em primitive} in a certain sense. Let us recall that basic disc classes consist of the following two types:
\begin{enumerate}
\item A disc class $\beta_j \in \pi_2(\mathcal{X},L)$ represented by a smooth holomorphic disc of Maslov index two corresponding to each toric divisor $D_j$.
\item A disc class $\beta_\nu \in \pi_2(\mathcal{X},L)$ represented by a holomorphic orbi-disc of desingularized Maslov index zero, with one interior orbifold marked point which maps to the twisted sector $\mathcal{X}_\nu$.
\end{enumerate}

Note that if $\beta\in\pi_2(\mathcal{X},L)$ is realized by a stable holomorphic (orbi-)disc, then we can write
$$\beta = \sum_{a=1}^p k_a \beta_{j_a} + \sum_{\nu \in\mathrm{Box}'} k_\nu \beta_\nu +\alpha,$$
for $k_a \in \N$, $k_\nu \in \integer_{\geq 0}$, $j_a\in \{1,\ldots, m\}$ and  $\alpha$ is an element in $H_2(\mathcal{X};\Z)$ realized by a positive sum of holomorphic (orbi-)spheres.
\begin{lemma} \label{LEM_prim-basic}
The basic disc classes are primitive, in the following sense:
\begin{itemize}
\item[(i)]
For $j \in \{1,\ldots, m\}$, suppose that $\beta_j = \sum_{a=1}^p k_a \beta_{j_a} + \sum_{\nu \in \mathrm{Box}'} k_\nu \beta_\nu +\alpha$ as above. Then one of the following alternatives holds:
\begin{enumerate}
\item At least one $k_\nu \geq 1$.
\item $k_\nu = 0$ for all $\nu \in \mathrm{Box}'$, $\alpha = 0$, $p=1$, $k_1=1$ and $j_1=j$.
\end{enumerate}

\item[(ii)]
For $\eta \in \mathrm{Box}'$, suppose that $\beta_\eta = \sum_{a=1}^p k_a \beta_{j_a} + \sum_{\nu \in \mathrm{Box}'} k_\nu \beta_\nu +\alpha$ as above. Then one of the following alternatives holds.
\begin{enumerate}
\item $\sum_\nu k_\nu \geq 2$.
\item $p = 0$, $\alpha=0$, and $k_\nu=1$ when $\nu=\eta$ and zero otherwise.
\end{enumerate}
\end{itemize}
\end{lemma}
\begin{proof}
For toric manifolds, a similar statement has been proved in \cite[Theorem 10.1]{FOOO1}.

Let us first consider the case of $\beta_j$ for $j \in \{1,\ldots, m\}$. We will assume that $k_\nu = 0$ for all $\nu \in \mathrm{Box}'$ and show that $\alpha =0$, $p=1$, $k_1=1$ and $j_1=j$. Since $k_\nu = 0$ for all $\nu \in \mathrm{Box}'$, one has
\begin{equation} \label{sum:a}
\beta_j =\sum_{a=1}^p k_a \beta_{j_a}+ \alpha.
\end{equation}
By considering the symplectic areas on both sides, we have
$$\ell_j = \sum_a k_a \ell_{j_a} + c,$$
where $c$ is the symplectic area of $\alpha$. Since $\alpha$ is represented by a positive sum of holomorphic (orbi-)spheres, we have $c \geq 0$ and equality holds if and only if $\alpha = 0$. On the other hand, take $u \in P$ in the interior of the $j$-th facet $F_j\subset P$ so that $\ell_j(u)=0$, and $\ell_i(u) >0$ for $i \neq j$. Hence, we must have $c \leq 0$. So $c=0$, $a=1$, $k_1=1$ and $j_1=j$. This proves (i).

To prove (ii), consider $\beta_\eta$ for some $\eta \in \mathrm{Box}'$, and assume that $\sum_\nu k_\nu < 2$. Again by taking the symplectic areas, one has
$$ \ell_\eta = \sum_{a=1}^p k_a \ell_{j_a} + \sum_{\nu \in \mathrm{Box}'} k_\nu \ell_\nu + c $$
where $c \geq 0$ is the symplectic area of $\alpha$. Take $u \in P$ such that $\ell_\eta(u) = 0$. Since every term on the right hand side is non-negative, we must have $\ell_{j_a}(u) = \ell_\nu(u) = c = 0$ for all $a$ and $\nu$. In particular, $\alpha = 0$. Also, this implies that $\eta, \bb_{j_a}, \nu$ for $k_\nu \neq 0$ belong to the same cone of the fan $\Sigma(P)$. But
$$\eta = \sum_{a=1}^p k_a \bb_{j_a} + \sum_{\nu \in \mathrm{Box}'} k_\nu \nu \in \mathrm{Box}.$$
This forces $k_a = 0$ for all $a$ and $k_\nu$ cannot be all zero. The only remaining possibility is $k_\nu = 1$ only when $\nu = \eta$ and zero otherwise.  This finishes the proof of the lemma.
%
%
%
%
%
\end{proof}

This lemma implies that the basic holomorphic (orbi-)discs cannot degenerate into sums of other basic discs, since the number of interior orbifold marked points cannot increase when we consider possible degenerations of an orbifold curve (by the definition of the topology of the domain curve).

We are now in a position to prove Theorem \ref{Thm_open_closed2}:
\begin{proof}[Proof of Theorem \ref{Thm_open_closed2}]
First of all, the semi-Fano condition implies that every rational (orbi-)curve has $c_1^{CW} \geq 0$. Since a sphere which intersects with $\mathcal{X}^\circ\cong(\C^*)^n \subset \mathcal{X}$ must have positive $c_1^{CW}>0$, those with $c_1^{CW}=0$ are contained in the toric divisors.  Moreover the basic disc class $\beta_0$ is primitive by Lemma \ref{LEM_prim-basic}.  Thus the domain of a stable disc $u$ representing $\beta_0 + d$ must be of the form $\bD$ with a rational (orbi-)curve $\mathcal{C}$ attached, where $u_0:=u|_\bD$ is a holomorphic (orbi-)disc representing $\beta_0$ and $u|_\mathcal{C}$ represents $d$ which has $c_1^{CW}(d)=0$.  The (orbi-)nodal point where $u_0$ is attached with $u|_\mathcal{C}$ lies in $S$, and so $\mathcal{C}$ must pass through $S$.  By the assumption, such rational (orbi-)curves in $\mathcal{X}$ are in one-to-one correspondence with those in $\bar{\mathcal{X}}$. As a result, we can apply the construction in Definition \ref{Defn_barX} and extend the arguments in the proof of Theorem \ref{THM_open_closed} to the current situation to show that
$$\CM^{\mathrm{op}}_{1,l}(\beta,\bx;p) = \CM^{\mathrm{cl}}_{l+1}(\bar{\beta},\bar{\bx};p)$$
as spaces with Kuranishi structures. Hence we obtain the desired equality.
\end{proof}

Notice that starting with a toric manifold $X$, in order to apply the open-closed relation discussed in this section, unavoidably one has to work with toric orbifolds since in general $\bar{X}$ is an orbifold (see Definition \ref{Defn_barX}).  In the manifold case, Theorem \ref{Thm_open_closed2} says the following:

\begin{corollary}
Let $X$ be a semi-Fano toric manifold (possibly non-compact) and $L$ a Lagrangian torus fiber of $X$.  Let $\beta = \beta_0 + d \in \pi_2(X,L)$, where $\beta_0$ is a basic disc class and $d$ is represented by a rational curve with $c_1(d) = 0$.

Assume that each component of a rational curve $C \subset \bar{X}$ representing $d$ has $c_1^{CW} = 0$, and if $C$ intersects $D_0$, then $C$ is contained in $\bar{X}\setminus D_{\infty}$. (Here, $D_0$ is the toric divisor that $\beta_0$ intersects, $\bar{X}$ is the toric orbifold constructed in Definition \ref{Defn_barX} which is birational to $X$, and $D_{\infty}$ is the divisor corresponding to $\bb_\infty$ involved in the construction.)
Then we have the equality
$$ n_{1,0,\beta}^X ([\mathrm{pt}]_{L}) = \langle [\mathrm{pt}]_{\bar{X}}\rangle^{\bar{X}}_{0,1,\bar{\beta}}$$
where $[\mathrm{pt}]_{L} \in H^n (L;\rat)$ denotes the point class of $L$ and $[\mathrm{pt}]_{\bar{X}} \in H^{2n}(\bar{X};\rat)$ denotes the point class of $\bar{X}$.
\end{corollary}

\section{Example: $\mathcal{X}=\proj(1,\ldots,1,n)$ and $Y=\proj(K_{\proj^n}\oplus\mathcal{O}_{\proj^n})$}\label{examples}

In this section, we prove the open crepant resolution conjecture (Conjecture \ref{open_CRC}) for the weighted projective space $\mathcal{X}=\proj(1,\ldots,1,n)$ which is Gorenstein and Fano, and whose crepant resolution is given by the semi-Fano toric manifold $Y=\proj(K_{\proj^{n-1}}\oplus\mathcal{O}_{\proj^{n-1}})$.

\subsection{Computation of open orbifold GW invariants}

The weighted projective space $\mathcal{X}=\proj(1,\ldots,1,n)$ is a toric orbifold described by the simplicial fan $\Sigma$ whose generators of rays are given by
\begin{align*}
\bb_1 = & (1,0,\ldots,0,0), \bb_2 = (0,1,\ldots,0,0), \ldots, \bb_{n-1} = (0,0,\ldots,0,1,0),\\
\bb_n = & (-1,-1,\ldots,-1,n), \bb_{n+1}=(0,0,\ldots,0,-1) \in N=\Z^n.
\end{align*}
There is a unique isolated orbifold point with $\Z_n$-singularity which corresponds to the cone generated by $\bb_1,\bb_2,\ldots,\bb_n$. The twisted sectors of $\mathcal{X}$ are hence given by the trivial one $\mathcal{X}_0 = \mathcal{X}$ together with the non-trivial ones $\mathcal{X}_{k/n}$ corresponding to
$$\nu_{k/n}:=\frac{k}{n}\sum_{j=1}^n \bb_j=(0,0,\ldots,0,k)\in N$$
for $k = 1, \ldots, n-1$, which are all supported at the isolated orbifold point in $\mathcal{X}$. The degree shifting number of $\mathcal{X}_{k/n}$ is given by
$$\iota_{k/n} := \iota(\nu_{k/n}) = k.$$
The weighted projective space $\mathcal{X}$ is Gorenstein (as $\iota_{k/n}$ is an integer for all $k$) and Fano.

Let $L\subset\mathcal{X}$ be a Lagrangian torus fiber. By Cho-Poddar's classification of holomorphic orbi-discs \cite{CP} (see Theorem \ref{classification_orbidiscs}), there is a basic orbi-disc class $\beta_{1/n}\in \pi_2(\mathcal{X},L)$ with $\mu_{CW}(\beta_{1/n}) = 2$ which is represented by a holomorphic orbi-disc with one boundary marked point and one interior orbifold marked point passing through the twisted sector $\mathcal{X}_{1/n}$. Note that $\mathcal{X}_{1/n}$ is the only non-trivial twisted sector with degree shifting number equal to one. Let $\mathbf{1}_{1/n}\in H^0(\mathcal{X}_{1/n}) \subset H^2_\mathrm{orb}(\mathcal{X})$ be the fundamental class of $\mathcal{X}_{1/n}$. We are interested in computing the open orbifold GW invariants $n_{1,l,\beta_{1/n}}^\mathcal{X}([\mathrm{pt}]_{L};\mathbf{1}_{1/n},\ldots,\mathbf{1}_{1/n})$.

By applying the construction in Definition \ref{Defn_barX} with $\beta=\beta_{1/n}$, we have $\bar{\mathcal{X}}=\mathcal{X}=\proj(1,\ldots,1,n)$ since $\bb_\infty=-\nu_{1/n}=\bb_{n+1}$, and $\bar{\beta}=\beta_{1/n}+\beta_{n+1}$ where $\beta_{n+1}\in\pi_2(\mathcal{X},L)$ is the smooth basic disc class corresponding to $\bb_{n+1}$. Now Theorem \ref{THM_open_closed} gives the equality
$$n_{1,l,\beta}^\mathcal{X}([\mathrm{pt}]_{L};\mathbf{1}_{1/n},\ldots,\mathbf{1}_{1/n}) = \langle[\mathrm{pt}]_{\mathcal{X}},\mathbf{1}_{1/n},\ldots,\mathbf{1}_{1/n}\rangle^{\mathcal{X}}_{0, l+1,\bar{\beta}}.$$

To compute the GW invariant $\langle[\mathrm{pt}]_{\mathcal{X}},\mathbf{1}_{1/n},\ldots,\mathbf{1}_{1/n}\rangle^{\mathcal{X}}_{0,l+1,\bar{\beta}}$, we will use the method developed in \cite{CLT11} adapted to the orbifold setting. Roughly speaking, this goes as follows. The invariants we need are encoded as a certain coefficient of the small $J$-function of $\mathcal{X}$. By applying the mirror theorem for orbifolds (i.e. Theorem \ref{closed_mirror_thm}( \cite{CCIT13}, \cite[Conjecture 4.3]{iritani09}), we can then compute the relevant coefficient and hence the invariants using the explicit and combinatorially defined $I$-function.

Recall that the small $J$-function of $\mathcal{X}$ is given by
\begin{align*}
J^\mathcal{X}(q,z) = & \conste^{\tau_{0,2}/z}\left(1 + \sum_\alpha \sum_{\stackrel{(d,l)\neq(0,0)}{d\in H_2^\mathrm{eff}(\mathcal{X})}} \frac{q^d}{l!} \langle 1,\tau_\mathrm{tw}, \ldots, \tau_\mathrm{tw}, \frac{\phi_\alpha}{z-\psi} \rangle^\mathcal{X}_{0,l+2,d} \phi^\alpha \right),
\end{align*}
where $\log q=\tau = \tau_{0,2} + \tau_\mathrm{tw} \in H^2_\mathrm{orb}(\mathcal{X})$ with $\tau_{0,2}=\tau_1\bar{p}_1 \in H^2(\mathcal{X})$ and $\tau_\mathrm{tw} = \tau_2\mathbf{1}_{1/n} \in H^0(\mathcal{X}_{1/n})$, and $q^d = \conste^{\langle\tau_{0,2},d\rangle}=q_1^{\langle \bar{p}_1,d\rangle}$.

The $H^0$-part of the coefficient of $1/z^2$ of $J^\mathcal{X}(q,z)$ is given by
\begin{align*}
\sum_{\stackrel{(d,l)\neq(0,0)}{d\in H_2^\mathrm{eff}(\mathcal{X})}}
\frac{q^d}{l!}\langle[\mathrm{pt}]_{\mathcal{X}},\tau_\mathrm{tw},\ldots,\tau_\mathrm{tw}\rangle^\mathcal{X}_{0,l+1,d}
= & q^{\bar{\beta}}\sum_{l \geq 0}\frac{1}{l!}\langle [\mathrm{pt}]_{\mathcal{X}},\tau_\mathrm{tw},\ldots,\tau_\mathrm{tw}\rangle^\mathcal{X}_{0,l+1,\bar{\beta}}\\
= & q^{\bar{\beta}}\sum_{l \geq 0}\frac{\tau_2^l}{l!}\langle [\mathrm{pt}]_{\mathcal{X}},\mathbf{1}_{1/n},\ldots,\mathbf{1}_{1/n}\rangle^\mathcal{X}_{0,l+1,\bar{\beta}}
\end{align*}
where the last equality follows because $\langle [\mathrm{pt}]_{\mathcal{X}}, \tau_\mathrm{tw}, \ldots, \tau_\mathrm{tw} \rangle_{0,l+1,d} \neq 0$ only when $c_1^{CW}(d)=2$ for dimension reasons, and $d = \bar{\beta}$ is the only curve class satisfying $c_1^{CW}(d) = 2$.  Thus the invariants $\langle [\mathrm{pt}]_{\mathcal{X}},\mathbf{1}_{1/n},\ldots,\mathbf{1}_{1/n}\rangle^\mathcal{X}_{0,l+1,\bar{\beta}}$ are contained in the $H^0$-part of the coefficient of $1/z^2$ of the small $J$-function.

The toric mirror theorem allows one to compute the $J$-function from the combinatorial data which defines $\mathcal{X}$ as follows. The extended stacky fan \cite{jiang08} of $\mathcal{X}$ can be defined by the exact sequence
$$ 0 \to \mathbb{L} \to \Z^{n+2} \to N \to 0 $$
where the homomorphism $\Z^{n+2} \to N$ is given by sending the standard basic vector $e_i$ to $\bb_i$ for $i=1,\ldots,n+1$ and $e_{n+2}$ to $\bb_{n+2}:=\nu_{1/n} = (0,\ldots,0,1)$. One has
$$
(\bb_1 \,\, \ldots \,\, \bb_{n-1} \,\, \bb_n \,\, \bb_{n+1} \,\, \bb_{n+2})
\left( \begin{array}{cc}
1 & 0 \\
\vdots & \vdots \\
1 & 0 \\
n & 1 \\
0 & 1
\end{array}
\right) = 0
$$
which defines the inclusion of the kernel $\mathbb{L}\cong\Z^2 \to \Z^{n+2}$. Let
$$d_1=\sum_{j=1}^n e_j+ne_{n+1},\ d_2=e_{n+1}+e_{n+2}$$
be a basis of $\mathbb{L}$. Then $H_2(\mathcal{X};\rat)$ is the subspace $\rat d_1\subset \mathbb{L}\otimes\rat$ and $\bar{\beta}=d_1/n$. Let
$$D_1=\ldots=D_n=(1,0),\ D_{n+1}=(n,1)\textrm{ and }D_{n+2}=(0,1)\in\mathbb{L}^\vee$$
denote the row vectors in the above matrix. Then $H^2(\mathcal{X};\rat)$ is the quotient $\mathbb{L}^\vee\otimes\rat/\rat D_{n+2}$. For $j=1,\ldots,n+1$, the image of $D_j$ in $H^2(\mathcal{X};\rat)$ is the Poincar\'e dual of the corresponding toric divisor; while the image of $D_{n+2}$ in  $H^2(\mathcal{X};\rat)$ is zero.

The secondary fan is supported in $\mathbb{L}^\vee_\R\cong\R^2$, and its rays are generated by $D_1=\ldots=D_n$, $D_{n+1}$ and $D_{n+2}$.  The secondary fan parametrizes stability conditions of the GIT quotients of $\C^{n+2}$ by $(\C^*)^2$ whose action is defined by the above exact sequence.  It consists of two cones, namely $\langle D_1,D_{n+1}\rangle_{\R_{\geq 0}}$ and $\langle D_{n+1},D_{n+2}\rangle_{\R_{\geq 0}}$.  When we choose the stability condition $\eta\in\langle D_{n+1},D_{n+2}\rangle_{\R_{>0}}$, the GIT quotient we obtain is the orbifold $\mathcal{X}$. When we choose the stability condition $\eta\in\langle D_1,D_{n+1}\rangle_{\R_{>0}}$, the GIT quotient we obtain is the crepant resolution $Y=\proj(K_{\proj^{n-1}}\oplus\mathcal{O}_{\proj^{n-1}})$.

The cone $\mathbb{K}_\mathrm{eff}$ is given by the subset
\begin{align*}
\mathbb{K}_\mathrm{eff}  = \left\{ \frac{a-b}{n}d_1+bd_2 \in\mathbb{L}\otimes\rat: a, b \in \Z_{\geq 0} \right\}.
\end{align*}
For $d=\frac{a-b}{n} d_1+b d_2\in \mathbb{K}_\mathrm{eff}$, $\nu(d)=\{\frac{b-a}{n}\}\sum_{j=1}^n\bb_j=\{\frac{b-a}{n}\}(0,\ldots,0,n)\in\mathrm{Box}$. So $\nu(d)=0$ if and only if $a \equiv b$ (mod $n$). Recall that the $I$-function (which takes values in $H^*_\mathrm{orb}(\mathcal{X})$) is defined as
\begin{align*}
I_\mathcal{X}(y,z)= & \conste^{\bar{p}_1\log y_1/z}\\
& \quad\left(\sum_{d\in\mathbb{K}_\mathrm{eff}}y^d\frac{\prod_{j:\langle D_j,d\rangle<0}\prod_{k\in[\langle D_j,d\rangle,0)\cap\Z}(\bar{D}_j+(\langle D_j,d\rangle-k)z)}{\prod_{j:\langle D_j,d\rangle>0}\prod_{k\in[0,\langle D_j,d\rangle)\cap\Z}(\bar{D}_j+(\langle D_j,d\rangle-k)z)}\mathbf{1}_{\nu(d)}\right),
\end{align*}
where $y^d=y_1^{\langle p_1,d\rangle}y_2^{\langle p_2,d\rangle}=y_1^{\frac{a-b}{n}}y_2^b$ if $d=\frac{a-b}{n} d_1+b d_2$ and $\mathbf{1}_{\nu(d)}\in H^0(\mathcal{X}_{\nu(d)})\subset H^{2\iota(\nu(d))}_\mathrm{orb}(\mathcal{X})$ is the fundamental class of the twisted sector $\mathcal{X}_{\nu(d)}$.

Let us compute the $H^0$-part of the coefficient of $1/z^2$ of $I_\mathcal{X}(y,z)$. Let $d=\frac{a-b}{n} d_1+b d_2\in \mathbb{K}_\mathrm{eff}$. For the coefficient of $y^d$ to have an image in $H^0(\mathcal{X})$, we need to have $\mathbf{1}_{\nu(d)}\in H^0_\mathrm{orb}(\mathcal{X})$ which is true if and only if $a \equiv b$ (mod $n$), and also $\langle D_j, d \rangle \not\in \Z_{<0}$ for all $j \neq n+2$ since otherwise the numerator of that term is a multiple of the Poincar\'e dual of $D_i$ which cannot lie in $H^0(\mathcal{X})$ (note that the image of $D_{n+2}$ in $H^2(\mathcal{X})$ is zero). This implies that
$$a-b\geq 0, a\geq 0.$$
In this case, the exponent of $z$ in the expression of the $I$-function is given by
$$-\sum_{j=1}^{n+2}\lceil\langle D_j,d\rangle\rceil=-\left(n\lceil\frac{a-b}{n}\rceil+a+b\right)=-\left(n\frac{a-b}{n}+a+b\right)=-2a,$$
which contributes to $1/z^2$ only when $a=1$. This in turn implies $b=1$. Hence the $H^0$-part of the coefficient of $1/z^2$ of $I_\mathcal{X}(y,z)$ is given by $y_2=y^{d_2}$.

The toric mirror theorem states that
\begin{equation*}
J_\mathcal{X}(q,z)=I_\mathcal{X}(y(q),z),
\end{equation*}
where $y=y(q)$ is the inverse of the mirror map $q=q(y)$. In particular, the $H^0$-parts of their $1/z^2$-coefficient are equal. Thus
$$y_2 = q_1^{1/n}\sum_{l \geq 0}\frac{\tau_2^l}{l!}\langle [\mathrm{pt}]_{\mathcal{X}},\mathbf{1}_{1/n},\ldots,\mathbf{1}_{1/n}\rangle^\mathcal{X}_{0,l+1,\bar{\beta}},$$
since we have $q^{\bar{\beta}} = \conste^{\langle\tau_{0,2},d_1/n\rangle}= q_1^{\langle \bar{p}_1,d_1/n\rangle} = q_1^{1/n}$. Note that this implies that the Lagrangian Floer superpotential $W^{LF}_\mathcal{X}$ is convergent.

Let us also compute the mirror map, which is given by the $H^2_\mathrm{orb}$-part of the $1/z$-coefficient of the $I$-function. Let $d=\frac{a-b}{n} d_1+b d_2\in \mathbb{K}_\mathrm{eff}$. The coefficient of $y^d$ to contributes to $H^2_\mathrm{orb}$, either when $\nu(d)=0$ or $\nu(d)=\nu_{1/n}$. But $\nu(d)=0$ if and only if $a\equiv b$ (mod $n$) in which case the exponent of $z$ is $-2a$, so this will not be part of the mirror map. Thus we must have $\nu(d)=\nu_{1/n}$ which is the case if and only if $b-a\equiv 1$ (mod $n$). Write $b=a+kn+1$ for $k\in\Z$. Then the exponent of $z$ is given by
$$-\sum_{j=1}^{n+2}\lceil\langle D_j,d\rangle\rceil=-\left(n\lceil\frac{a-b}{n}\rceil+a+b\right)=n(-k)+a+(a+kn+1)=2a+1,$$
which contributes to $1/z$ only when $a=0$. So $b=kn+1$ for $k\in\Z_{\geq0}$. Hence the mirror map is given by
$$\tau(y_1,y_2) = p_1\log y_1 + \left(\sum_{k=0}^\infty \frac{((-\frac{1}{n})(-\frac{1}{n}-1)\cdots(-\frac{1}{n}-(k-1)))^n}{(kn+1)!}y^{d_k}\right)\mathbf{1}_{1/n},$$
where $d_k=(kn+1)(d_2-d_1/n)=-(k+1/n)d_1+(kn+1)d_2$. We can also write
\begin{align*}
\tau_1 & = \log y_1,\\
\tau_2 & = g(y_1^{-1/n}y_2),
\end{align*}
where $g=g(z)$ is the function
$$g(z):=\sum_{k=0}^\infty \frac{((-\frac{1}{n})(-\frac{1}{n}-1)\cdots(-\frac{1}{n}-(k-1)))^n}{(kn+1)!}z^{kn+1}.$$
We remark that $g$ is a solution to some Picard-Fuchs equation.

Using these calculations, we can now prove the open toric mirror theorem for $\mathcal{X}$:
\begin{theorem}\label{open_mirror_thm_wPn}
For $\mathcal{X}=\proj(1,\ldots,1,n)$, we have
\begin{equation*}
W^{LF}_\mathcal{X}(q)=W^{HV}_\mathcal{X}(y(q)),
\end{equation*}
where $y=y(q)$ is the inverse mirror map.
\end{theorem}
\begin{proof}
The Lagrangian Floer superpotential of $\mathcal{X}$ is given by
$$W^{LF}_\mathcal{X}(q)=Z_1+\ldots+Z_{n+1}+\left(\sum_{l \geq 0}\frac{\tau_2^l}{l!} n^\mathcal{X}_{1,l,\beta_{1/n}}([\mathrm{pt}]_L;\mathbf{1}_{1/n},\ldots,\mathbf{1}_{1/n})\right)Z_{n+2},$$
where $Z_j=C_jz^{\bb_j}$ and the coefficients $C_j$ are subject to the constraints
\begin{align*}
C_1\cdots C_n C_{n+1}^n & = q_1,\textrm{ and}\\
C_{n+1}C_{n+2} & = q^{d_1/n}=q_1^{1/n}.
\end{align*}
Letting $W_j=Z_j$ for $j=1,\ldots,n+1$ and
$$W_{n+2}=\left(\sum_{l \geq 0}\frac{\tau_2^l}{l!} n^\mathcal{X}_{1,l,\beta_{1/n}}([\mathrm{pt}]_L;\mathbf{1}_{1/n},\ldots,\mathbf{1}_{1/n})\right)Z_{n+2},$$
we can write $W_j=C_j'z^{\bb_j}$ where now the coefficients $C_j'$ are subject to the constraints
$$C_1'\cdots C_n'(C_{n+1}')^n=C_1\cdots C_n C_{n+1}^n=q_1=y_1,$$
and
\begin{align*}
C_{n+1}'C_{n+2}' & = q_1^{1/n}\left(\sum_{l \geq 0}\frac{\tau_2^l}{l!} n^\mathcal{X}_{1,l,\beta_{1/n}}([\mathrm{pt}]_L;\mathbf{1}_{1/n},\ldots,\mathbf{1}_{1/n})\right)\\
& = q_1^{1/n}\left(\sum_{l \geq 0}\frac{\tau_2^l}{l!}
\langle [\mathrm{pt}]_{\mathcal{X}},\mathbf{1}_{1/n},\ldots,\mathbf{1}_{1/n}\rangle^\mathcal{X}_{0,l+1,\bar{\beta}}\right)\\
& = y_2
\end{align*}
This shows that $W^{LF}_\mathcal{X}(q)=W^{HV}_\mathcal{X}(y(q))$.
\end{proof}

To get explicit numbers, let $f$ be the inverse function of $g$, so that $y_1^{-1/n}y_2 = f(\tau_2)$. Thus
$$f(\tau_2) = \sum_{l\geq0}\frac{\tau_2^l}{l!}n^\mathcal{X}_{1,l,\beta}([\mathrm{pt}]_L;\mathbf{1}_{1/n},\ldots,\mathbf{1}_{1/n}).$$

For $n = 2$, the inverse function $f(\tau_2)$ is simply $2 \sin \tau_2/2$.
Hence
$$
\sum_{l\geq0}\frac{\tau_2^l}{l!}n_{1,l,\beta}^\mathcal{X}([\mathrm{pt}]_L;\mathbf{1}_{1/n},\ldots,\mathbf{1}_{1/n})= 2 \sin \tau_2/2 = \sum_{j\geq 0} \frac{(-1)^j \tau_2^{2j+1}}{(2j+1)! 2^{2j}}
$$
and we get
$$
n^\mathcal{X}_{1,l,\beta}([\mathrm{pt}]_L;\mathbf{1}_{1/2},\ldots,\mathbf{1}_{1/2}) = \left\{
\begin{array}{ll}
0 & \textrm{ when } l \textrm{ is even;} \\
\frac{(-1)^j}{2^{2j}} & \textrm{ when } l=2j+1 \textrm{ for } j \in \Z_{\geq 0}.
\end{array}
\right.
$$

For $n=3$, one may compute the Taylor series expansion of the inverse function $f$ and obtain:
\begin{align*}
n^\mathcal{X}_{1,l,\beta}([\mathrm{pt}]_L;\mathbf{1}_{1/3},\ldots,\mathbf{1}_{1/3})=
\left\{\begin{array}{cl}
0 & \textrm{ when } l \not\equiv 1 \mod 3; \\
1 & \textrm{ when } l=1;\\ \vspace{1pt}
\frac{1}{27} & \textrm{ when } l=4;\\ \vspace{1pt}
\frac{29}{729} & \textrm{ when } l=7;\\ \vspace{1pt}
\frac{6607}{19683} & \textrm{ when } l=10;\\ \vspace{1pt}
\frac{4736087}{531441} & \textrm{ when } l=13;\\ \vspace{1pt}
\frac{7710586801}{14348907} & \textrm{ when } l=16;\\ \vspace{1pt}
\vdots & \hspace{10pt} \vdots
\end{array}
\right.
\end{align*}

\subsection{Open CRC}

By the results of \cite{CLLT11}, \cite{CLLT12} (see also \cite{CLT11}), the open mirror theorem (Theorem \ref{open_mirror_thm_Y}) is true for the semi-Fano toric manifold $Y=\proj(K_{\proj^{n-1}}\oplus\mathcal{O}_{\proj^{n-1}})$ {\em without} any convergence assumption. By our discussion in Subsection \ref{closedCRC_vs_openCRC}, the open CRC (Conjecture \ref{open_CRC}) would then follow from the existence of an analytic continuation of the mirror map for $Y$, which in turn is implied by the existence of the symplectic transformation $\mathbb{U}$ that appeared in the closed CRC (see Theorem \ref{closedCRC_implies_openCRC}). Using Mellin-Barnes integrals \cite{BH06}, one can indeed show that the analytic continuation of the mirror map for $Y$ exists, hence proving Conjecture \ref{open_CRC} for $\mathcal{X}=\proj(1,\ldots,1,n)$ and $Y=\proj(K_{\proj^{n-1}}\oplus\mathcal{O}_{\proj^{n-1}})$.


The mirror map for $Y$ is given by
\begin{align*}
Q_1 & = U_1\exp(nH(U_1)),\\
Q_2 & = U_2\exp(-H(U_1)),
\end{align*}
where $H=H(z)$ is the function
\begin{equation*}
H(z)=\sum_{k\geq1}(-1)^{kn}\frac{(kn-1)!}{(k!)^n}z^k.
\end{equation*}
Here $Q=(Q_1,Q_2)$ are coordinates on the $A$-model moduli space $\CM_A^Y:=H^2(Y;\C^*)$ of $Y$ and $U=(U_1,U_2)$ are coordinates on the $B$-model moduli $\CM_B^Y:=H^2(Y;\C^*)$ of $Y$.

On the other hand, as shown in the previous subsection, the mirror map for $\mathcal{X}$ is given by
\begin{align*}
q_1 & = y_1,\\
\tau_2 & = g(y_1^{-1/n}y_2),
\end{align*}
where $g=g(z)$ is the function
\begin{align*}
g(z) & = \sum_{k=0}^\infty \frac{((-\frac{1}{n})(-\frac{1}{n}-1)\cdots(-\frac{1}{n}-(k-1)))^n}{(kn+1)!}z^{kn+1}\\
& = \sum_{k=0}^\infty \frac{(-1)^{kn}}{(kn+1)!}\left(\frac{\Gamma(k+1/n)}{\Gamma(1/n)}\right)^n z^{kn+1},
\end{align*}
where $(\tau_1=\log q_1,\tau_2)$ are coordinates on the $A$-model moduli space $\CM_A^\mathcal{X}$ of $\mathcal{X}$ and $y=(y_1,y_2)$ are coordinates on the $B$-model moduli $\CM_B^\mathcal{X}$ of $\mathcal{X}$.

The B-model moduli spaces $\CM_B^Y$ and $\CM_B^\mathcal{X}$ can be glued together using the secondary fan for $\mathcal{X}$ which is spanned by the vectors $D_1=(1,0)$, $D_{n+1}=(n,1)$ and $D_{n+2}=(0,1)$. The vectors $D_1, D_{n+1}$ are dual to the coordinates $U_1,U_2$ on $\CM_B^Y$. Let $\eta_1,\eta_2$ be the coordinates  on $\CM_B^\mathcal{X}$ dual to the vectors $D_{n+2},D_{n+1}$. Then these two coordinate systems are related by
\begin{equation*}
\eta_1=U_1^{-1/n},\ \eta_2=U_1^{1/n}U_2;
\end{equation*}
or
\begin{equation*}
U_1=\eta_1^{-n},\ U_2=\eta_1\eta_2.
\end{equation*}
Since the coordinates $\eta_1$, $\eta_2$ correspond to the generators $d_2-d_1/n$ and $d_1/n$ of $\mathbb{K}_\mathrm{eff}$ respectively, they are related to the original coordinates $y_1$, $y_2$ (which correspond to $d_1$, $d_2$ respectively) by
\begin{equation*}
\eta_1=y_1^{-1/n}y_2,\ \eta_2=y_1^{1/n};
\end{equation*}
or
\begin{equation*}
y_1=\eta_2^n,\ y_2=\eta_1\eta_2.
\end{equation*}
Altogether, the coordinate systems $(y_1,y_2)$ on $\CM_B^\mathcal{X}$ and $(U_1,U_2)$ on $\CM_B^Y$ are related by
\begin{equation*}
y_1=U_1U_2^n,\ y_2=U_2;
\end{equation*}
or
\begin{equation*}
U_1=y_1y_2^{-n},\ U_2=y_2.
\end{equation*}

Using Mellin-Barnes integral (see \cite{BH06}), one can analytically continue the function $\log Q_1(U_1)$ from places where $|U_1|$ is small to places where $|U_1|$ is large (and hence $|\eta_1=y_1^{-1/n}y_2|$ is small) to obtain the function $\Lambda(y)$. The results are as follows:

When $n$ is even, $\log Q_1$ is analytically continued to
\begin{align*}
\sum_{l=1}^{n-1}\frac{(-1)^l\pi e^{-l\pi\mathbf{i}/n}}{\Gamma(1-l/n)^n\sin(l\pi/n)} \sum_{k\geq0}\frac{(-1)^{nk}}{(nk+l)!}\left(\frac{\Gamma(k+l/n)}{\Gamma(l/n)}\right)^n(y_1^{-1/n}y_2)^{nk+l},
\end{align*}
while $\log Q_2=\frac{1}{n}(\log y_1-\log Q_1)$. When $n=2$, by choosing a suitable branch cut, this gives
\begin{align*}
\log Q_1 & = -\mathbf{i}(\pi-g(y_1^{-1/2}y_2)),\\
\log Q_2 & = \frac{1}{2}\log y_1+\frac{\mathbf{i}}{2}(\pi-g(y_1^{-1/2}y_2)),
\end{align*}
which yields the change of variables:
$$Q_1=\conste^{-\mathbf{i}(\pi-\tau_2)},\ Q_2=q_1^{1/2}\conste^{\mathbf{i}(\pi-\tau_2)/2}.$$
Since this is an affine linear change of coordinates, it preserves the flat structures in the neighborhoods $U_\mathcal{X}$ and $U_Y$ near the large radius limit points corresponding to $\mathcal{X}=\proj(1,1,2)$ and $Y=\proj(K_{\proj^1}\oplus\mathcal{O}_{\proj^1})=\mathbb{F}_2$. As shown in \cite{CIT09}, the Frobenius manifolds defined by the genus 0 Gromov-Witten theory for $\mathcal{X}$ and $Y$ are in fact isomorphic, and this is due to the fact that the weighted projective plane $\mathcal{X}=\proj(1,1,2)$ satisfies the Hard Lefschetz condition.

When $n$ is odd, $\log Q_1$ is analytically continued to
\begin{align*}
\sum_{l=1}^{n-1}\frac{(-1)^l\pi}{\Gamma(1-l/n)^n\sin(l\pi/n)} \sum_{k\geq0}\frac{(-1)^{nk}}{(nk+l)!}\left(\frac{\Gamma(k+l/n)}{\Gamma(l/n)}\right)^n(y_1^{-1/n}y_2)^{nk+l},
\end{align*}
while $\log Q_2=\frac{1}{n}(\log y_1-\log Q_1)$. In particular, the flat structures near the large radius limit points for $\mathcal{X}$ and $Y$ are not preserved. When $n=3$, this is given by
$$-\frac{2\sqrt{3}\pi}{3\Gamma(\frac{2}{3})^3}g(y_1^{-1/3}y_2)+\frac{2\sqrt{3}\pi}{3\Gamma(\frac{1}{3})^3}
\sum_{k\geq0}\frac{(-1)^{k}}{(3k+2)!}\left(\frac{\Gamma(k+\frac{2}{3})}{\Gamma(\frac{2}{3})}\right)^3(y_1^{-1/3}y_2)^{3k+2}$$
which agrees with the results in \cite[Section 3.9]{CIT09}.

As an immediate consequence of the existence of analytic continuation of the mirror map for $Y$, we have
\begin{theorem}\label{open_CRC_wPn}
The open crepant resolution conjecture (i.e. Conjecture \ref{open_CRC}) holds for the weighted projective space $\mathcal{X}=\proj(1,\ldots,1,n)$ and its crepant resolution $Y=\proj(K_{\proj^{n-1}}\oplus\mathcal{O}_{\proj^{n-1}})$.
\end{theorem}

\bibliographystyle{amsplain}
\bibliography{geometry}

\end{document}